\documentclass[12pt]{amsart}

\usepackage{epic,eepic,epsfig,amssymb,amsmath,amsthm,graphics,enumerate}

\usepackage{multimedia}

\usepackage{pdftricks}

\begin{psinputs}
\usepackage{pstricks,pst-text}
\end{psinputs}

\usepackage{vmargin}
\usepackage{a4wide}
\usepackage{supertabular}
\usepackage{array}
\usepackage{multicol}

\setmarginsrb{3cm}{2cm}{3cm}{3cm}{75pt}{20pt}{20pt}{30mm}
\def\R{\mathbb R}
\def\N{\mathbb N}

\def\Z{\mathbb Z}
\def\P{\mathbb P}

\def\J{\mathbb J}

\def\H{{\rm H}}
\def\G{\Gamma}
\def\Rm{{\rm Rm}}

\newtheorem{theorem}{Theorem}[section]
\newtheorem{lemma}[theorem]{Lemma}
\newtheorem{proposition}[theorem]{Proposition}

\theoremstyle{definition}
\newtheorem{definition}[theorem]{Definition}
\newtheorem{remark}[theorem]{Remark}

\def\signal{\bigskip\begin{center} {\sc Ayadi Lazrag\par\vspace{3mm}
Univ. Nice Sophia Antipolis,\\ CNRS, LJAD, UMR 7351\\ 06100 Nice\\ FRANCE\par\vspace{3mm}
email:} \tt{Ayadi.Lazrag@unice.fr}\end{center}}

\def\signlr{\bigskip\begin{center} {\sc Ludovic Rifford\par\vspace{3mm}
Univ. Nice Sophia Antipolis\\
\& Institut Universitaire de France\\ CNRS, LJAD, UMR 7351\\ 06100 Nice\\ FRANCE\par\vspace{3mm}
email:} \tt{ludovic.rifford@math.cnrs.fr}\end{center}}

\def\signrr{\bigskip \begin{center} {\sc Rafael O. Ruggiero\par\vspace{3mm}
PUC-Rio, Departamento de Matem\'{a}tica. Rua Marqu\'{e}s de S\~{a}o Vicente 225, G\'{a}vea,
22450-150, Rio de Janeiro\\ BRAZIL\par\vspace{3mm}
email:} \tt{rorr@mat.puc-rio.br}\end{center}}

\begin{document}

\title[Frank's lemma for $C^2$-Ma\~n\'e perturbations]{Franks' lemma for $C^2$-Ma\~n\'e perturbations of Riemannian metrics and applications to persistence}

\author{A. Lazrag}
\author{L. Rifford}
\author{R. Ruggiero}

\begin{abstract}
Given a compact Riemannian manifold, we prove a uniform Franks' lemma at second order for geodesic flows and apply the result in persistence theory.
\end{abstract}

\maketitle

\section{Introduction}
One of the most important tools of $C^{1}$ generic and stability theories of dynamical systems is the celebrated Franks Lemma \cite{franks71}: 
\bigskip

Let $M$ be a smooth ({\it i.e.} of class $C^{\infty}$) compact manifold of dimension $n\geq 2$ and let  $f:M \longrightarrow M$ be a $C^1$ diffeomorphism.  Consider a finite set of points $S= \{p_{1}, p_{2},.., p_{m}\}$, let $\Pi = \bigoplus_{i=1}^{m} T_{p_{i}}M$, 
$\Pi ' = \bigoplus_{i=1}^{m} T_{f(p_{i})}M$. Then there exist $\epsilon_{0} >0$ such that for every $0<\epsilon \leq \epsilon_{0}$ there exists $\delta = \delta(\epsilon) >0$ such that the following holds:\\

\noindent Let $L = (L_{1}, L_{2},..,L_{m}): \Pi \longrightarrow \Pi '$ be an isomorphism such that 
$$
\bigl\| L_{i} - D_{p_{i}}f \bigr\|<\delta \qquad \forall i = 1, \ldots, m,
$$
then there exists a $C^1$ diffeomorphism $g :M \longrightarrow M$ satisfying 
\begin{enumerate}
\item $g(p_{i}) = f(p_{i})$ for every $i=1, \ldots, m$,
\item $ D_{p_{i}}g = L_{i}$ for each $i=1, \ldots, m$, 
\item the diffeomorphim $g$ is in the $\epsilon$ neighborhood of $f$ in the $C^{1}$ topology. 
\end{enumerate}
\bigskip

In a few words, the lemma asserts that given a collection $S$  of $m$ points $p_{i}$ in the manifold $M$, any isomorphism 
from $\Pi$ to $\Pi '$ can be the collection of the differentials of a diffeomorphism $g$, $C^{1}$ 
close to $f$, at each point of $S$ provided that the isomorphism is sufficiently close to the direct sum of the 
maps $D_{p_{i}}f$, $i=1, \ldots,m$. The sequence of points is particularly interesting for applications in dynamics 
when the collection $S$ is a subset of a periodic orbit. The idea of the proof of the lemma is quite elementary: we conjugate the 
isomorphisms $L_{i}$ by the exponential map of $M$ in suitably small neighborhoods of the points $p_{i}$'s 
and then glue (smoothly) the diffeomorphism $f$ outside the union of such neighborhoods with these collection of 
conjugate-to-linear maps. So the proof strongly resembles an elementary calculus exercise: we can glue a $C^{1}$  
function $h: {\mathbb R} \longrightarrow {\mathbb R}$ outside a small neighborhood $U$ of a point $x$ with 
the linear function in $U$ whose graph is the line through $(x,h(x))$ with slope $h'(x)$ and get a new function 
that is $C^{1}$ close to $h$. 

The Franks lemma admits a natural extension to flows, and its important applications in 
the study of stable dynamics gave rise to versions for more specific families of systems, like symplectic diffeomorphisms and 
Hamiltonian flows \cite{kn:Robinson,kn:Vivier}. It is clear that for specific families of systems the proof of the lemma should be more 
difficult that just gluing conjugates of linear maps by the exponential map since this surgery procedure in 
general does not preserve specific properties of systems, like preserving symplectic forms in the case of symplectic maps. 
The Frank's Lemma was extensively used by R. Ma\~{n}\'{e} in his proof of the $C^{1}$ structural stability conjecture \cite{mane88}, and 
we could claim with no doubts that it is one of the pillars of the proof together with C. Pugh's $C^{1}$ closing lemma \cite{pugh67a,pugh67b}
(see Newhouse \cite{kn:Newhouse} for the proof of the $C^{1}$ structural stability conjecture for symplectic diffeomorphisms).

A particularly challenging problem is to obtain a version of Frank's Lemma for geodesic flows. First of all, a typical 
perturbation of the  geodesic flow of a Riemannian metric in the family of smooth flows is not the geodesic flow of another Riemannian metric. To ensure that perturbations of a geodesic flow are geodesic flows as well the most natural way to proceed is to perturb the Riemannian metric in the manifold itself. But then, since a local perturbation 
of a Riemannian metric changes all geodesics through a neighborhood, the geodesic flow of the perturbed metric 
changes in tubular neighborhoods of vertical fibers in the unit tangent bundle. So local perturbations of the metric are not 
quite local for the geodesic flow, the usual strategy applied in generic dynamics of perturbing a flow in a flowbox without 
changing the dynamics outside the box does not work. This poses many interesting, technical problems in the theory of local perturbations of dynamical systems of geometric origin, the famous works of Klingenberg-Takens \cite{kt72} and Anosov \cite{anosov82} (the bumpy metric theorem) about generic properties of closed geodesics are perhaps the two best known examples.  Moreover, geodesics in general have many self-intersections so the effect of a local perturbation of the metric on the global dynamics of perturbed orbits is unpredictable unless we know a priori that the geodesic flow enjoys some sort of stability (negative sectional curvatures, Anosov flows for instance). 

The family of metric perturbations which preserves a compact piece of a given geodesic  is the most used to study generic theory of periodic geodesics. This family of perturbations is relatively easy to characterize analytically when we restrict ourselves to the category of conformal perturbations or more generally, to the set of perturbations of Lagrangians by small 
potentials. Recall that a Riemannian metric $h$ in a manifold $M$ is conformally equivalent to a Riemannian metric $g$ in $M$ 
if there exists a positive, $C^{\infty}$ function $b: M \longrightarrow{\mathbb R}$ such that $h_{x}(v,w) = b(x)g_{x}(v,w)$ 
for every $x \in M$ and $v,w \in T_{x}M$. Given a $C^{\infty}$, Tonelli Lagrangian $L : TM\times TM \longrightarrow {\mathbb R}$ 
defined in a compact manifold $M$, and a $C^{\infty}$ function $u : M\longrightarrow {\mathbb R}$, the function 
$L_{u}(p,v) = L(p,v) + u(p)$ gives another Tonelli Lagrangian. The function $u$ is usually called a potential because of the 
analogy between this kind of Lagrangian and mechanical Lagrangians.

 By Maupertuis principle (see for example \cite{dfn92}), the Lagrangian associated to a metric $h$ in $M$ that is conformally equivalent to $g$ is of the form $L(p,v)= \frac{1}{2}g_{p}(v,v) + u(p)$ 
for some function $u$. Since the Lagrangian of a metric $g$ is given by the formula $L_{g}(p,v) = \frac{1}{2}g_{p}(v,v)$, 
we get $L_{h}(p,v) = L_{g}(p,v) + u(p)$. Now, given a compact part $\gamma: [0,T] \longrightarrow M$ 
of a geodesic of $(M,g)$, the collection of potentials $u: M\rightarrow {\mathbb R}$ such that $\gamma[0,T]$ 
is still a geodesic of $L(p,v) = L_{g}(p,v) +u(p)$ contains the functions whose gradients vanish along the subset of 
$T_{\gamma(t)}M$ which are perpendicular to $\gamma'(t)$ for every $t \in [0,T]$ (see for instance \cite[Lemma 2.1]{ruggiero91-2}). Lagrangian perturbations 
of Tonelli Lagrangians of the type $L_{h}(p,v) = L_{g}(p,v) + u(p)$ were used extensively by R. Ma\~{n}\'{e} to study generic properties of Tonelli Lagrangians and applications to Aubry-Mather theory (see for instance \cite{kn:Mane1,kn:Mane2}). Ma\~{n}\'{e}'s idea proved to be 
very fruitful and insightful in Lagrangian generic theory, and opened a new branch of generic theory 
that is usually called Ma\~{n}\'{e}'s genericity. Recently, Rifford-Ruggiero \cite{rr12} gave a proof of Klingenberg-Takens and 
Anosov $C^{1}$ genericity results for closed geodesics using control theory techniques applied to the class of Ma\~{n}\'{e} type 
perturbations of Lagrangians. Control theory ideas simplify a great deal the technical problems involved in metric perturbations 
and at the same time show that Ma\~{n}\'{e} type perturbations attain full Hamiltonian genericity. This result, combined with a previous theorem by Oliveira \cite{oliveira08} led to the  Kupka-Smale Theorem for geodesic flows in the family of conformal perturbations of metrics. 

These promissing applications of control theory to the generic theory of geodesic flows motivate us to study Frank's Lemma for conformal perturbations of Riemannian metrics or equivalently, for Ma\~{n}\'{e} type perturbations of Riemannian Lagrangians. 
Before stating our main theorem, let us recall first some notations and 
basic results about geodesic flows. The geodesic flow of 
a Riemannian manifold $(M,g)$ will be denoted by $\phi_{t}$, the flow acts on the unit tangent bundle $T_{1}M$, a point 
$\theta \in T_{1}M$ has canonical coordinates $\theta = (p,v)$ where $p \in M$, $v \in T_{p}M$, and  
$\gamma_{\theta}$ denotes the unit speed geodesic with initial conditions $\gamma_{\theta}(0)=p$, $\gamma_{\theta}'(0)=v$. 
Let $N_{\theta} \subset T_{\theta}T_{1}M$ be the plane of vectors which are perpendicular to the geodesic flow with respect to the Sasaki metric  (see for example \cite{sakai}). The collection of these planes is preserved by the action of the differential of the geodesic flow:$ D_{\theta}\phi_{t}(N_{\theta}) = N_{\phi_{t}(\theta)}$ for every $\theta $ and $t \in {\mathbb R}$.

Let us consider a geodesic arc, of length $T$
$$
\gamma_{\theta} : \left[0,T\right] \longrightarrow M,
$$
and let $\Sigma_0$ and  $\Sigma_T$ be local  transverse sections for the geodesic flow which are tangent to 
$N_{\theta}$ and $N_{\phi_{T}(\theta)}$ respectively.  Let $\P_g(\Sigma_0,\Sigma_T,\gamma)$ be a Poincar\'e map going from $\Sigma_0$ to $\Sigma_T$.  In horizontal-vertical coordinates of $N_{\theta}$, the differential $D_{\theta}\phi_{T}$ that is 
the \textit{linearized Poincar\'e map}
$$
P_g(\gamma)(T):=D_{\theta}\P_g(\Sigma_0,\Sigma_T,\gamma)
$$
is a symplectic endomorphism of $\mathbb{R}^{(2n-2)} \times \mathbb{R}^{(2n-2)}$. This endomorphism 
can be expressed in terms of the Jacobi fields of $\gamma_{\theta}$ which are perpendicular to $\gamma_{\theta}'(t)$ for every $t$: 
$$
P_g(\gamma)(T)(J(0),\dot{J}(0))=(J(T),\dot{J}(T)),
$$
where $\dot{J}$ denotes
the covariant derivative along the geodesic. We can identify the set of all symplectic endomorphisms of $\mathbb{R}^{2n-2} \times \mathbb{R}^{2n-2}$ with the symplectic group 
$$
\mbox{Sp}(n-1):= \Bigl\{X \in \R^{(2n-2) \times (2n-2)} ; X^*\J X=\J \Bigr\},  
$$
where $X^*$ denotes the transpose of $X$ and
$$
\J= \left[ \begin{matrix}
0&I_{n-1}\\
-I_{n-1}&0\\
\end{matrix} \right].
$$
Given a geodesic $\gamma_{\theta}:[0,T] \rightarrow M$, an interval $[t_1,t_2] \subset [0,T]$ and $\rho>0$, we denote by $\mathcal{C}_g\left(\gamma_{\theta}\bigl( [t_1,t_2]\bigr);\rho\right)$ the open geodesic cylinder along $\gamma_{\theta}\bigl( [t_1,t_2]\bigr)$ of radius $\rho$, that is the open set defined by
\begin{multline*}
\mathcal{C}_g\left(\gamma_{\theta}\bigl( [t_1,t_2]\bigr);\rho\right) := \\
\Bigl\{ p \in M \, \vert \, \exists t\in (t_1,t_2) \mbox{ with } d_g\bigl(p,\gamma_{\theta}(t)\bigr)<\rho \mbox{ and }  d_g\bigl(p,\gamma_{\theta}([t_1,t_2])\bigr) = d_g\bigl(p,\gamma_{\theta}(t)\bigr) \Bigr\},
\end{multline*}
where $d_g$ denotes the geodesic distance with respect to $g$.  Our main result is the following.

\begin{theorem}[Franks' Lemma] \label{THMmain}
Let $(M,g)$ be a smooth compact Riemannian manifold of dimension $\geq 2$. For every $T>0$ there exist $\delta_{T}, \tau_T, K_T>0$ such that the following property holds:\\
For every geodesic $\gamma_{\theta}:[0,T] \rightarrow M$,  there are $\bar{t}\in [0,T-\tau_T]$ and $\bar{\rho}>0$ with 
$$
\mathcal{C}_g\Bigl( \gamma_{\theta}\left( \bigl[ \bar{t},\bar{t}+\tau_T\bigr]\right) ;\bar{\rho} \Bigr) \cap \gamma_{\theta}([0,T]) = \gamma_{\theta}\left( \bigl( \bar{t},\bar{t}+\tau_T\bigr)\right),
$$
such that for every $\delta \in (0, \delta_{T})$, for each symplectic map $A$ in the open ball  (in $\mbox{Sp}(n-1)$) centered at  $P_g(\gamma)(T)$ of radius $\delta$ and for every $\rho \in (0,\bar{\rho})$, there exists a $C^{\infty}$ metric $h$ in $M$ that is conformal to $g$, $h_{p}(v,w) = (1+\sigma(p))g_{p}(v,w)$, such that: 
\begin{enumerate}
\item the geodesic $\gamma_{\theta} : [0,T] \longrightarrow M$ is still a geodesic of $(M,h)$, 
\item $\mbox{Supp} (\sigma) \subset \mathcal{C}_g\left( \gamma_{\theta}\left( \bigl[ \bar{t},\bar{t}+\tau_T\bigr]\right) ;\rho \right)$,
\item $P_{h}(\gamma_{\theta})(T) = A$,
\item the $C^{2}$ norm of the function $\sigma$ is less than $K_T \sqrt{\delta}$.
\end{enumerate}
\end{theorem}

Theorem \ref{THMmain} improves a previous result by Contreras \cite[Theorem 7.1]{contreras10} which gives a controllability result at first order under an additional assumption on the curvatures along the initial geodesic. Other proofs of Contreras Theorem can also be found in \cite{visscher14} and \cite{lazrag14}. The Lazrag proof follows already the ideas from geometric control introduced in \cite{rr12} to study controllability properties at first order. Our new Theorem \ref{THMmain} shows that controllability holds at second order without any assumption on curvatures along the geodesic. Its proof amounts to study how small conformal perturbations of the metric $g$ along $\Gamma:=\gamma([0,T])$ affect the differential of $\P_g(\Sigma_0,\Sigma_T,\gamma)$. This can be seen as  a problem of local controllability along a reference trajectory in the symplectic group. As in \cite{rr12}, The idea is to see the Hessian of the conformal factor along the initial geodesic as a control and to obtain Theorem \ref{THMmain} as a uniform controllability result at second order for a control system of the form
$$
\dot{X}(t)  = A(t) X(t) +  \sum_{i=1}^k u_i(t) B_i  X(t), \qquad \mbox{for a.e.} \quad  t,
$$
in the symplectic group $\mbox{Sp}(n-1)$. \\

We apply Franks' Lemma  to extend some results concerning the characterization of hyperbolic geodesic flows in terms 
of the persistence of some $C^{1}$ generic properties of the dynamics. These results are based on well known steps towards the 
proof of the $C^{1}$ structural stability conjecture for diffeomorphisms. 

Let us first introduce some notations. Given a smooth compact Riemannian manifold $(M,g)$, we say that a property $P$ of the 
geodesic flow of $(M,g)$ is $\epsilon$-$C^{k}$-persistent from Ma\~{n}\'{e}'s viewpoint if for every $C^{\infty}$ function $f: M\longrightarrow {\mathbb R}$ whose $C^{k}$ norm is less than $\epsilon$ we have that the geodesic flow of the metric $(M,(1+f)g)$ has property $P$ as well.  By Maupertuis' principle, this is equivalent to the existence of an open $C^{k}$-ball of radius $\epsilon'>0$ of functions $q: M \longrightarrow {\mathbb R}$ such that for every $C^{\infty}$ function in this open ball the 
Euler-Lagrange flow of the Lagrangian $L(p,v) = \frac{1}{2}g_{p}(v,v) - q(p)$ in the level of energy equal to 1 has property $P$.  This definition is inspired by the definition of $C^{k-1}$ persistence for diffeomorphisms: a property $P$ of a diffeomorphism 
$f:M\longrightarrow M$ is called $\epsilon$-$C^{k-1}$ persistent if the property holds for every diffeomorphism in the 
$\epsilon$-$C^{k-1}$ neighborhood of $f$. It is clear that if a property $P$ is $\epsilon$-$C^{1}$ persistent for a geodesic flow then the property $P$ is 
$\epsilon'$-$C^{2}$ persistent from Ma\~{n}\'{e}'s viewpoint for some $\epsilon'$.

\begin{theorem} \label{THM1}
Let $(M,g)$ be a smooth compact Riemannian manifold of dimension $\geq 2$ such that the periodic orbits of the geodesic flow are $C^{2}$-persistently hyperbolic from
Ma\~{n}\'{e}'s viewpoint. Then the closure of the set of periodic orbits of the geodesic flow is a hyperbolic set.
\end{theorem}

An interesting application of Theorem \ref{THM1} is the following extension of Theorem A in \cite{ruggiero91}:
$C^{1}$ persistently expansive geodesic flows in the set of Hamiltonian flows of $T_{1}M$ are Anosov flows. We recall that a non-singular smooth flow $\phi_{t}: Q \longrightarrow  Q$ acting on a
complete Riemannian manifold $Q$ is {\it $\epsilon$-expansive} if given $x \in Q$ we have that for each $y \in Q$ such that there exists a continuous surjective function
$\rho: {\mathbb R} \longrightarrow {\mathbb R}$ with $\rho(0)=0$ satisfying
$$ 
d \left(\phi_{t}(x), \phi_{\rho(t)}(y)\right) \leq \epsilon \qquad \forall t \in \R, 
$$
for every $ t \in {\mathbb R}$ then there exists $t(y)$, $\mid t(y) \mid <\epsilon$ such that $\phi_{t(y)}(x)=y$.
A smooth non-singular flow is called expansive if it is expansive for some $\epsilon >0$. Anosov flows are expansive, and it is not difficult to get examples which show that the converse of this statement is not true. Theorem \ref{THM1} yields the following.

\begin{theorem} \label{THM2}
Let $(M,g)$ be a smooth compact Riemannian manifold, suppose that either $M$ is a surface or $\dim M \geq 3$ and $(M,g)$ has no conjugate points. Assume that the geodesic flow is $C^{2}$ persistently expansive from Ma\~{n}\'{e}'s viewpoint, then the geodesic flow is Anosov.
\end{theorem}

The proof of the above result requires the set of periodic orbits to be dense. Such a result follows from expansiveness on surfaces \cite{ruggiero91} and from the absence of conjugate points in any dimension. If we drop the assumption of the absence of conjugate points we do not know whether periodic orbits of expansive geodesic flows 
are dense (and so if the geodesic flow in Theorem \ref{THM2} is Anosov). This is a difficult, challenging problem. \\

The paper is organized as follows. In the next section, we introduce some preliminaries which describe the relationship between local controllability and some properties of the End-Point mapping and we introduce the notions of local controllability at first and second order. We recall a result of controllability at first order (Proposition \ref{LIEPROP1}) already used in \cite{rr12} and state results (Propositions \ref{LIEPROP3} and \ref{LIEPROP3bis}) at second order whose long and technical proofs are given in Sections \ref{ProofLIEPROP3} and \ref{ProofLIEPROP3bis}. In Section \ref{proofTHMmain}, we provide the proof of Theorem \ref{THMmain} and the proof of theorems \ref{THM1}, \ref{THM2} are given in Section \ref{proofTHM1THM2}.  

\section{Preliminaries in control theory}\label{prelcontrol}

Our aim here is to provide sufficient conditions for first and second order local controllability results. This kind of results could be developed for nonlinear control systems on smooth manifolds. 
For sake of simplicity, we restrict our attention here to the case of affine control systems on the set of (symplectic) matrices. We refer the interested reader to \cite{as04,coron07,lazragthesis,jurjevic97,riffordbook} for a further study in control theory.

\subsection{The End-Point mapping}

Let  us a consider a \textit{bilinear control system} on $M_{2m}(\R)$ (with $m, k \geq 1$), of the form
\begin{eqnarray}
\label{syscontrol}
\dot{X}(t)  = A(t) X(t) +  \sum_{i=1}^k u_i(t) B_i  X(t), \qquad \mbox{for a.e.} \quad  t,
\end{eqnarray}
where the \textit{state} $X(t)$ belongs to $M_{2m}(\R)$, the \textit{control} $u(t)$ belongs to $\R^k$, $t \in [0,T] \mapsto A(t)$ (with $T>0$) is a smooth map valued in $M_{2m}(\R)$, 
and $B_1, \ldots, B_k$ are $k$ matrices in $M_{2m}(\R)$. Given $\bar{X}\in M_{2m}(\R)$ and $\bar{u} \in L^2 \bigl([0,T]; \R^k\bigr)$, the Cauchy problem
\begin{eqnarray}
\label{cauchyX}
\dot{X} (t) = A(t) X (t) + \sum_{i=1}^k \bar{u}_i(t) B_i X(t) \quad \mbox{for a.e.} \quad  t \in [0,T], \qquad X(0) = \bar{X},
\end{eqnarray}
possesses a unique solution $X_{\bar{X},\bar{u}}(\cdot)$. The \textit{End-Point mapping} associated with $\bar{X}$ in time $T>0$ is defined as
$$
\begin{array}{rcl}
E^{\bar{X},T} \, : \, L^2 \bigl([0,T]; \R^k\bigr)  & \longrightarrow & M_{2m}(\R) \\
u & \longmapsto & X_{\bar{X},u}(T).
\end{array}
$$
It is a smooth mapping whose differential can be expressed in terms of the linearized control system (see \cite{riffordbook}). Given $\bar{X}\in M_{2m}(\R)$, $\bar{u} \in L^2 \bigl([0,T]; \R^k\bigr)$, and setting $\bar{X} (\cdot) :=  X_{\bar{X},\bar{u}}(\cdot)$, the differential of $E^{\bar{X}} $ at $\bar{u}$ is given by the 
linear operator
$$
\begin{array}{rccc}
D_{\bar{u}}E^{\bar{X},T}  \, : \, &L^2\bigl([0,T]; \R^k\bigr) & \longrightarrow & M_{2m}(\R) \\
&v & \longmapsto & Y(T),
\end{array}
$$
where $Y(\cdot)$ is the unique solution to the linearized Cauchy problem
\begin{eqnarray*}
\label{linearized}
\left\{
\begin{array}{l}
\dot{Y}(t) = A(t) Y(t) + \sum_{i=1}^k v_i(t) B_i(t) \bar{X}(t) \quad \mbox{for a.e.} \quad t \in [0,T],\\
Y(0)= 0.
\end{array}
\right.
\end{eqnarray*}
Note that if we denote by $S(\cdot)$ the solution to the Cauchy problem
\begin{equation}
\label{eq:St}
\left\{
\begin{array}{l}
\dot{S}(t) = A(t) S(t)\\
S(0)=I_{2m}
\end{array}
\right.
\qquad \forall t \in [0,T],
\end{equation}
then there holds
\begin{eqnarray}\label{raab}
 D_{\bar{u}}E^{\bar{X},T}  (v) =  \sum_{i=1}^k S(T) \int_0^{T} v_i(t) S(t)^{-1} B_i \bar{X}(t) \,dt,
\end{eqnarray}
for every $v \in L^2([0,T]; \R^k)$. \\

Let $\mbox{Sp}(m)$ be the symplectic group in $M_{2m}(\R)$ ($m\geq 1$), that is the smooth submanifold of matrices $X \in M_{2m}(\R)$ satisfying
$$
X^* \J X = \J \quad  \mbox{ where } \J := \left[ \begin{matrix} 0 & I_m \\ -I_m & 0 \end{matrix} \right].
$$
Denote by $\mathcal{S}(2m)$ the set of symmetric matrices in $M_{2m}(\R)$. The tangent space to  $\mbox{Sp}(m)$ at the identity matrix is given by
$$
T_{I_{2m}} \mbox{Sp}(m) = \Bigl\{ Y \in M_{2m}(\R) \, \vert \, \J Y \in \mathcal{S}(2m)  \Bigr\}.
$$
Therefore, if there holds
\begin{eqnarray}\label{condsymp}
\J A (t), \, \J B_1, \, \ldots, \, \J B_k \in \mathcal{S}(2m) \qquad \forall t \in [0,T],
\end{eqnarray}
then $\mbox{Sp}(m)$ is invariant with respect to (\ref{syscontrol}), that is for every $\bar{X}\in \mbox{Sp}(m)$ and $\bar{u} \in L^2 \bigl([0,T]; \R^k\bigr)$,
$$
X_{\bar{X},u}(t) \in \mbox{Sp}(m) \qquad \forall t \in [0,T].
$$
In particular, this means that for every $\bar{X} \in \mbox{Sp}(m)$, the End-Point mapping $E^{\bar{X},T}$ is valued in $\mbox{Sp}(m)$. Given $\bar{X}\in \mbox{Sp}(m)$ and $\bar{u} \in L^2 \bigl([0,T]; \R^k\bigr)$, 
we are interested in local controllability properties of (\ref{syscontrol}) around $\bar{u}$. The control system  (\ref{syscontrol}) is called \textit{controllable around} $\bar{u}$ in $\mbox{Sp}(m)$ (in time $T$) if 
for every final state $X\in \mbox{Sp}(m)$ close to $X_{\bar{X},\bar{u}}(T)$ there is a control $u\in L^2 \bigl([0,T]; \R^k\bigr)$ which steers $\bar{X}$ to $X$, that is such that $E^{\bar{X},T}(u)=X$. 
Such a property is satisfied as soon as $E^{\bar{X},T}$ is locally open at $\bar{u}$. Our aim in the next sections is to give an estimate from above on the size of $\|u\|_{L^2}$ in terms of $\|X-X_{\bar{X},u}(T)\|$. 

\subsection{First order controllability results}

Given $T>0$, $\bar{X}\in \mbox{Sp}(m)$, a mapping $t\in [0,T] \mapsto A(t) \in M_{2m}(\R)$, $k$ matrices  $B_1, \ldots, B_k \in M_{2m}(\R)$ satisfying (\ref{condsymp}), and $\bar{u} \in L^2 \bigl([0,T]; \R^k\bigr)$, 
we say that the control system (\ref{syscontrol}) is  \textit{controllable at first order around} $\bar{u}$ in $\mbox{Sp}(m)$ if the mapping $E^{\bar{X},T} : L^2 \bigl([0,T]; \R^k\bigr) \rightarrow \mbox{Sp}(m)$ is a 
\textit{submersion } at $\bar{u}$, that is if the linear operator 
$$
D_{\bar{u}}E^{\bar{X},T} \, : \, L^2\bigl([0,T]; \R^k\bigr)  \longrightarrow  T_{\bar{X}(T)} \mbox{Sp}(m),
$$
is surjective (with $\bar{X}(T):=X_{\bar{X},\bar{u}}(T)$). The following sufficient condition for first order controllability is given in \cite[Proposition 2.1]{rr12} (see also \cite{lazrag14,lazragthesis}).

\begin{proposition}
\label{LIEPROP1}
Let $T>0$, $t\in [0,T] \mapsto A(t)$ a smooth mapping and $B_1, \ldots, B_k \in M_{2m}(\R)$ be matrices in $M_{2m}(\R)$ satisfying (\ref{condsymp}). Define the $k$ sequences of smooth mappings
$$
\{B_1^j\}, \ldots, \{B_k^j\} : [0,T] \rightarrow  T_{I_{2m}} \mbox{Sp}(m)
$$
by
\begin{eqnarray}\label{brackets}
\left\{
\begin{array}{l}
B_i^0(t) := B_i\\
B_i^j(t) := \dot{B}_{i}^{j-1}(t) + B_i^{j-1}(t) A(t)  - A(t) B_i^{j-1}(t),
\end{array}
\right.
\end{eqnarray}
for every $t\in [0,T]$ and every $i \in \{1, \ldots, k\}$. Assume that there exists some $\bar{t} \in [0,T]$ such that
\begin{eqnarray}\label{conditionLIEPROP}
\mbox{Span} \Bigl\{ B_i^j(\bar{t}) \, \vert \, i \in \{1,\ldots, k\}, j\in \N\Bigr\} = T_{I_{2m}} \mbox{Sp}(m).
\end{eqnarray}
Then for every $\bar{X} \in \mbox{Sp}(m)$, the control system (\ref{syscontrol}) is controllable at first order around $\bar{u}\equiv 0$.
\end{proposition}

The control system which is relevant in the present paper is not always controllable at first order. We need sufficient condition for controllability at second order.

\subsection{Second-order controllability results}

Using the same notations as above, we say that the control system (\ref{syscontrol}) is  \textit{controllable at second order around} $\bar{u}$ in $\mbox{Sp}(m)$ if there are $ \mu , K >0$ such that 
for every $X \in B \Bigl( \bar{X}(T) ,\mu \Bigr) \cap \mbox{Sp}(m)$, there is $u\in L^2 \bigl([0,T]; \R^k\bigr)$ satisfying 
$$
E^{\bar{X},T}(u) = X \quad \mbox{and} \quad \|u\|_{L^2} \leq K \left| X- \bar{X}(T)\right|^{1/2}.
$$
Obtaining such a property requires a study of the End-Point mapping at second order. Recall that given two matrices $B, B' \in M_{2m}(\R)$, the bracket $[B,B']$ is the matrix of $M_{2m}(\R)$ defined as
$$
[B,B'] := BB'-B'B.
$$
The following results are the key points in the proof of our main theorem. Their proofs will be given respectively in Sections \ref{ProofLIEPROP3} and \ref{ProofLIEPROP3bis}.

\begin{proposition}
\label{LIEPROP3}
Let $T>0$, $t\in [0,T] \mapsto A(t)$ a smooth mapping and $B_1, \ldots, B_k \in M_{2m}(\R)$ be matrices in $M_{2m}(\R)$ satisfying (\ref{condsymp}) such that 
\begin{eqnarray}
\label{ASSBB}
B_iB_j=0 \qquad \forall i,j=1, \ldots, k.
\end{eqnarray}
Define the $k$ sequences of smooth mappings 
$\{B_1^j\}, \ldots, \{B_k^j\} : [0,T] \rightarrow  T_{I_{2m}} \mbox{Sp}(m)$ by (\ref{brackets}) and assume that the following properties are satisfied with $\bar{t}=0$:
\begin{eqnarray}\label{conditionLIEPROP4}
 \left[B_i^j(\bar{t}),B_i \right] \in \mbox{Span}  \Bigl\{B_r^s(\bar{t}) \, \vert \, r =1,..,k, \, s \geq 0  \Bigr\} \qquad \forall i =1, \ldots, k, \, \forall j =1,2,
\end{eqnarray}
and
\begin{eqnarray}\label{conditionLIEPROP3}
\mbox{Span}  \Bigl\{B_i^j(\bar{t}),[B_i^1(\bar{t}),B_l^1(\bar{t})] \, \vert \, i,l = 1,..,k\, and \, j=0,1,2 \Bigr\}= T_{I_{2m}}\mbox{Sp}(m).
\end{eqnarray}
Then, for every $\bar{X} \in \mbox{Sp}(m)$, the control system (\ref{syscontrol}) is controllable at second order around $\bar{u}\equiv 0$. 
\end{proposition}

\begin{remark}
For sake of simplicity we restrict here our attention to control systems of the form (\ref{syscontrol}) satisfying (\ref{ASSBB})-(\ref{conditionLIEPROP4}). More general results can be found in \cite{lazragthesis}.
\end{remark}

To prove Theorem \ref{THMmain}, we will need the following parametrized version of Proposition \ref{LIEPROP3} which will follow from the fact that smooth controls with support in $(0,T)$ are dense in $L^2([0,T];\R^k)$ and compactness. 

\begin{proposition}
\label{LIEPROP3bis}
Let $T>0$, and for every $\theta $ in some set of parameters $\Theta$ let $t\in [0,T] \mapsto A^{\theta}(t)$ be a smooth mapping and $B_1^{\theta}, \ldots, B_k^{\theta} \in M_{2m}(\R)$ be matrices in $M_{2m}(\R)$ satisfying (\ref{condsymp}) (with $A(t)=A^{\theta}$) such that 
\begin{eqnarray}
\label{ASSBBbis}
B_i^{\theta}B_j^{\theta}=0 \qquad \forall i,j=1, \ldots, k.
\end{eqnarray}
Define for every $\theta \in \Theta$ the $k$ sequences of smooth mappings 
$\{B_1^{\theta,j}\}, \ldots, \{B_k^{\theta,j}\} : [0,T] \rightarrow  T_{I_{2m}} \mbox{Sp}(m)$ as in (\ref{brackets}) and assume that the following properties are satisfied with $\bar{t}=0$ for every $\theta\in \Theta$:
\begin{eqnarray}\label{conditionLIEPROP4bis}
\quad  \left[B_i^{\theta,j}(\bar{t}),B_i^{\theta} \right] \in \mbox{Span}  \Bigl\{B_r^{\theta,s}(\bar{t}) \, \vert \, r =1,..,k, \, s \geq 0  \Bigr\} \quad \forall i =1, \ldots, k, \, \forall j =1,2,
\end{eqnarray}
and
\begin{eqnarray}\label{conditionLIEPROP3bis}
\quad \mbox{Span}  \Bigl\{B_i^{\theta,j} (\bar{t}),[B_i^{\theta,1}(\bar{t}),B_l^{\theta,1}(\bar{t})] \, \vert \, i,l = 1,..,k\, and \, j=0,1,2 \Bigr\}= T_{I_{2m}}\mbox{Sp}(m).
\end{eqnarray}
Assume moreover, that the sets 
$$
\Bigl\{ B_i^{\theta} \, \vert \, i=1, \ldots, k, \, \theta \in \Theta \Bigr\} \subset M_{2m}(\R)
$$
and
$$
\Bigl\{ t\in [0,T] \mapsto A^{\theta}(t) \, \vert \, \theta \in \Theta \Bigr\} \subset C^2\bigl([0,T];M_{2m}(\R)\bigr)
$$
are compact. Then, there are $ \mu , K >0$ such that for every $\theta \in \Theta$, every $\bar{X}  \in \mbox{Sp}(m)$ and every $ X \in B \Bigl( \bar{X}^{\theta}(T) ,\mu \Bigr) \cap \mbox{Sp}(m)$ ($\bar{X}^{\theta}(T)$ denotes the solution at time $T$ of the control system (\ref{syscontrol}) with parameter $\theta$ starting from $\bar{X}$), there is $u\in C^{\infty} \bigl([0,T]; \R^k\bigr)$ with support in $[0,T]$ satisfying 
$$
E_{\theta}^{\bar{X},T}(u) = X \quad \mbox{and} \quad \|u\|_{C^2} \leq K \left| X- \bar{X}(T)\right|^{1/2}
$$
($E_{\theta}^{\bar{X},T}$ denotes the End-Point mapping associated with the control system (\ref{syscontrol}) with parameter $\theta$).
\end{proposition}

Our proof is based on a series of results on openness properties of $C^2$ mappings near critical points in Banach spaces which was developed by Agrachev and his co-authors, see \cite{as04}.

\subsection{Some sufficient condition for local openness}
Here we are interested in the study of mappings $F: \mathcal{U} \rightarrow \R^N$ of class $C^2$ in an open set $\mathcal{U}$ in some Banach space $X$. We call critical point of $F$ any $u\in \mathcal{U}$ such that $D_uF: \mathcal{U} \rightarrow \R^N$ is not surjective. We call corank of $u$, the quantity
$$
\mbox{corank} (u) := N - \mbox{dim} \left(  \mbox{Im} \bigl( D_uF \bigr)\right).
$$
If $Q:\mathcal{U} \rightarrow \R$ is a quadratic form, its negative index is defined by 
$$
\mbox{ind}_- (Q) := \max \Bigl\{ \mbox{dim} (L) \ \vert \ Q_{\vert L \setminus \{0\}} <0 \Bigr\}.
$$
The following non-quantitative result whose proof can be found in \cite{as04,lazragthesis,riffordbook} provides a sufficient condition at second order for local openness.

\begin{theorem}\label{THMopen}
Let $F: \mathcal{U} \rightarrow \R^N$ be a mapping of class $C^2$ on an open set $\mathcal{U} \subset X$ and $\bar{u} \in \mathcal{U}$ be a critical point of $F$ of corank $r$. If 
\begin{eqnarray}\label{ind}
\mbox{ind}_- \left( \lambda^* \left( D^2_{\bar{u}} F \right)_{\vert \mbox{Ker} (D_{\bar{u}} F)}  \right) \geq r \qquad \forall \lambda \in \left( \mbox{Im} \bigl( D_{\bar{u}} F \bigr)\right)^{\perp} \setminus \{0\},
\end{eqnarray}
then the mapping $F$ is locally open at $\bar{u}$, that is the image of any neighborhood of $\bar{u}$ is an neighborhood of $F(\bar{u})$.
\end{theorem}

In the above statement, $\left( D^2_{\bar{u}} F \right)_{\vert \mbox{Ker} (D_{\bar{u}} F)}$ refers to the quadratic mapping from $\mbox{Ker} (D_{\bar{u}} F)$ to $\R^N$ defined by
$$
\left( D^2_{\bar{u}} F \right)_{\vert \mbox{Ker} (D_{\bar{u}} F)} (v) :=  D^2_{\bar{u}} F \cdot (v,v) \qquad \forall v \in \mbox{Ker} (D_{\bar{u}} F).
$$
The following result is a quantitative version of the previous theorem. (We denote by $B_X(\cdot,\cdot)$ the balls in $X$ with respect to the norm $\|\cdot\|_X$.)

\begin{theorem}\label{THMopenquant}
Let $F: \mathcal{U} \rightarrow \R^N$ be a mapping of class $C^2$ on an open set $\mathcal{U} \subset X$ and $\bar{u} \in \mathcal{U}$ be a critical point of $F$ of corank $r$. Assume that (\ref{ind}) holds. Then there exist $\bar{\epsilon}, c  \in (0,1)$ such that for every $\epsilon \in (0,\bar{\epsilon})$ the following property holds: For every $u \in \mathcal{U}, z \in \R^N$ with
\begin{eqnarray}
\left\| u - \bar{u} \right\|_X < \epsilon, \quad \left| z- F(u) \right| < c\, \epsilon^2,
\end{eqnarray}
there are $w_1, w_2 \in X$ such that $u+w_1+w_2\in \mathcal{U}$, 
\begin{eqnarray}
z = F \bigl( u+ w_1 + w_2\bigr),
\end{eqnarray}
and
\begin{eqnarray}
w_1 \in  \mbox{Ker} \left(D_{u} F\right), \quad \bigl\| w_1\bigr\|_X < \epsilon, \quad \bigl\| w_2\bigr\|_X < \epsilon^2.
\end{eqnarray}
\end{theorem}

Again, the proof of Theorem \ref{THMopenquant} which follows from previous results by Agrachev-Sachkov \cite{as04} and Agrachev-Lee  \cite{al09} can be found in \cite{lazragthesis,riffordbook}. A parametric version  of Theorem \ref{THMopenquant} that will be useful in the proof of Proposition \ref{LIEPROP3bis} is provided in \cite{lazragthesis}.

\subsection{Proof of Proposition \ref{LIEPROP3}}\label{ProofLIEPROP3}

Without loss of generality, we may assume that $\bar{X}=I_{2m}$. As a matter of fact, if $X_u :[0,T] \rightarrow \mbox{Sp}(m) \subset M_{2m}(\R)$ is solution to the Cauchy problem
\begin{eqnarray}\label{cauchy10fev}
 \dot{X}_u (t) = A(t) X_u (t) + \sum_{i=1}^k u_i(t) B_i X_u(t)  \mbox{ for a.e. }   t \in [0,T], \quad X_u(0) = I_{2m},
\end{eqnarray}
then for every $\bar{X} \in  \mbox{Sp}(m)$, the trajectory $\left(X_u \bar{X}\right):[0,T] \rightarrow M_{2m}(\R)$ starts at $\bar{X}$ and satisfies
$$
\frac{d}{dt} \left( {X}_u (t)\bar{X}\right) = A(t) \left( X_u (t)\bar{X}\right) + \sum_{i=1}^k u_i(t) B_i \left( X_u(t)\bar{X}\right) \quad \mbox{for a.e.} \quad  t \in [0,T].
$$
So any trajectory of (\ref{syscontrol}), that is any control, steering $I_{2m}$ to some $X\in  \mbox{Sp}(m)$ gives rise to a trajectory, with the same control, steering $\bar{X}\in \mbox{Sp}(m)$ to $X\bar{X} \in \mbox{Sp}(m)$. Since right-translations in  $\mbox{Sp}(m)$ are diffeomorphisms, we infer that  local controllability at second order around $\bar{u}\equiv 0$ from $\bar{X}=I_{2m}$ implies controllability at second order around $\bar{u}\equiv 0$ for any $\bar{X} \in \mbox{Sp}(m)$. So from now we assume that $\bar{X}=I_{2m}$ (in the sequel we omit the lower index and simply write $I$). We recall that $\bar{X}:[0,T] \rightarrow \mbox{Sp}(m) \subset M_{2m}(\R)$ denotes the solution of (\ref{cauchy10fev}) associated with $u=\bar{u}\equiv0$ while $X_u:[0,T] \rightarrow \mbox{Sp}(m) \subset M_{2m}(\R)$ stands for a solution of (\ref{cauchy10fev}) associated with some control $u \in L^2 \bigl([0,T]; \R^k\bigr)$. Furthermore, we may also assume that the End-Point mapping $E^{I,T}:L^2  \bigl([0,T]; \R^k\bigr) \rightarrow \mbox{Sp}(m)$ is not a submersion at $\bar{u}$ because it would imply controllability at first order around $\bar{u}$ and so at second order, as desired. 

We equip the vector space $M_{2m}(\R)$ with the scalar product defined by
$$
P\cdot Q = \mbox{tr} \left( P^*Q \right) \qquad \forall P,Q \in M_{2m}(\R).
$$
Let us fix $P \in T_{\bar{X}(T)} \mbox{Sp}(m)$ such that $P$ belongs to $\left( \mbox{Im} \bigl( D_{0} E^{I,T} \bigr)\right)^{\perp} \setminus \{0\}$ with respect to our scalar product (note that $\left( \mbox{Im} \bigl( D_{0} E^{I,T} \bigr)\right)^{\perp} \setminus \{0\}$ is nonempty since $D_{0} E^{I,T} : L^2  \bigl([0,T]; \R^k\bigr) \rightarrow T_{\bar{X}(T)} \mbox{Sp}(m)$ is assumed to be not surjective).

\begin{lemma}\label{LEM20sept1}
For every $t\in [0,T]$, we have
$$
\mbox{tr} \Bigl[ P^* S(T)S(t)^{-1} \, B_i^j(t)S(t)  \Bigr] = 0 \qquad \forall j \geq 0, \forall i= 1,...,k.
$$
\end{lemma}

\begin{proof}[Proof of Lemma \ref{LEM20sept1}]
Recall (remember (\ref{raab})) that for every $u\in L^2([0,T];\R^k)$, 
$$
D_0E^{I,T} (u) = S(T) \int_0^T S(t)^{-1} \sum_{i=1}^k u_i(t) B_i \bar{X}(t) \, dt,
$$
where $S(\cdot)$ denotes the solution of the Cauchy problem (\ref{eq:St}). Thus if $P$ belongs to $ \left(\mbox{Im} D_0E^{I,T}\right)^{\perp}$, we have 
$$
\mbox{tr} \left[ P^*  S(T) \int_0^T S(t)^{-1} \sum_{i=1}^k u_i(t) B_i \bar{X}(t) \, dt\right] = 0 \qquad \forall u \in  L^1([0,T];\R^k).
$$
This can be written as
$$
 \sum_{i=1}^k \int_0^Tu_i(t) \,  \mbox{tr} \Bigl[ P^* S(T)  S(t)^{-1}  B_i S(t)\Bigr] \, dt = 0 \qquad \forall u \in  L^1([0,T];\R^k).
 $$
We infer that 
$$
 \mbox{tr} \Bigl[ P^* S(T)  S(t)^{-1}  B_i S(t)\Bigr]  = 0 \qquad \forall i \in \{1, \ldots, k\}, \, \forall t \in [0,T].
$$
We conclude by noticing that 
$$
\frac{d^j}{dt^j} \left( S(t)^{-1}  B_i S(t)\right) = S(t)^{-1}  B_i^j(t) S(t) \qquad \forall t \in [0,T].
$$
\end{proof}

Let $u\in L^2 \bigl([0,T]; \R^k\bigr)$ be fixed, for every $\epsilon \in \R$ small we define $\delta_{\epsilon} :[0,T] \rightarrow M_{2m}(\R)$ by 
$$
\delta_{\epsilon}(t) := E^{I,t}(\epsilon u) \qquad \forall t \in [0,T].
$$
By regularity of the End-Point mapping (see \cite{riffordbook}), we have formally for every $t\in [0,T]$,
$$
\delta_{\epsilon}(t) = \bar{X}(t) + \delta_{\epsilon}^1(t) + \delta_{\epsilon}^2(t) +o(\epsilon^2),
$$
where $ \delta_{\epsilon}^1$ is linear in $\epsilon$ and $\delta_{\epsilon}^2$ quadratic. Then we have for every $t\in [0,T]$,
\begin{eqnarray*}
\delta_{\epsilon}(t) &= &  \bar{X}(t) + \delta_{\epsilon}^1(t) + \delta_{\epsilon}^2(t) +o(\epsilon^2)  \\
 & = & I + \int_0^t A(s) \delta_{\epsilon}(s) + \sum_{i=1}^k \epsilon \, u_i(s) B_i \delta_{\epsilon}(s) \, ds\\
 & = &  \bar{X}(t) +  \int_0^t A(s) \delta_{\epsilon}^1(s) + \sum_{i=1}^k \epsilon \, u_i(s) B_i \bar{X}(s) \, ds \\
 &  & \qquad +  \int_0^t A(s) \delta_{\epsilon}^2(s) + \sum_{i=1}^k \epsilon \, u_i(s) B_i \delta_{\epsilon}^1(s) \, ds + o(\epsilon^2).
\end{eqnarray*}
Consequently, the second derivative of $E^{I,T}$ at $0$ is given by the solution (times $2$) at time $T$ of the Cauchy problem
$$
\left\{
\begin{array}{l}
\dot{Z}(t) = A(t) Z(t) + \sum_{i=1}^k \, u_i(t) B_i Y(t),\\
Z(0)=0,
\end{array}
\right.
$$
where $Y:[0,T]\rightarrow M_{2m}(\R)$ is solution to the linearized Cauchy problem (\ref{linearized}). Therefore we have
$$
D_0^2E^{I,T} (u) = 2 S(T) \int_0^T S(t)^{-1} \sum_{i=1}^k u_i(t) B_i \varphi(t) \, dt,
$$
where 
$$
\varphi(t) :=  \sum_{i=1}^k S(T) \int_0^T S(t)^{-1}  u_i(t) B_i \bar{X}(t) \, dt.
$$
Then we infer that  for every $u\in L^2([0,T];\R^k)$, 
\begin{multline}\label{30oct1}
P \cdot D_0^2E^{I,T} (u) = \\
2 \sum_{i,j=1}^k \int_0^T \int_0^t u_i(t) u_j(s)  \mbox{tr} \Bigl[ P^* S(T)  S(t)^{-1}  B_i S(t) S(s)^{-1} B_j S(s) \Bigr] \, ds \, dt.  
\end{multline}
It is useful to work with an approximation of the quadratic form $P\cdot  D_0^2E^{I,T}$. For every $\delta>0$, we see the space $L^2([0,\delta];\R^k)$ as a subspace of $L^2([0,T];\R^k)$ by the canonical immersion
$$
u \in L^2([0,\delta];\R^k) \, \longmapsto \, \tilde{u} \in L^2([0,T];\R^k), 
$$
with
$$
\tilde{u}(t) := \left\{ \begin{array}{cl} u(t) & \mbox{ if } t \in [0,\delta]\\
0 & \mbox{ otherwise}.
\end{array}
\right.
\qquad \mbox{for a.e. } t \in [0,T]. 
$$
For sake of simplicity, we keep the same notation for $\tilde{u}$ and $u$.

\begin{lemma}\label{LEM20sept2}
There is $C>0$ such that for every $\delta \in (0,T)$, we have 
$$
\Bigl| P \cdot D_0^2E^{I,T} (u) - Q_{\delta}(u) \Bigr| \leq C \delta^4 \, \|u\|_{L^2}^2 \qquad \forall u \in  L^2([0,\delta];\R^k) \subset L^2([0,T];\R^k),
$$
where $Q_{\delta} :  L^2([0,\delta];\R^k) \rightarrow \R$ is defined by
$$
Q_{\delta} (u) := 2 \sum_{i,j=1}^k \int_0^{\delta} \int_0^t u_i(t)u_j(s) \mathcal{P}_{i,j}(t,s)  \, ds \, dt \qquad \forall u \in L^2([0,\delta];\R^k),
$$
with
\begin{multline*}
  \mathcal{P}_{i,j}(t,s) = \mbox{tr} \left[ P^* S(T) \left(sB_iB_j^1(0)+tB_i^1(0)B_j+ \frac{s^2}{2}B_iB_j^2(0) \right. \right.\\
  \left. \left. +\frac{t^2}{2}B_i^2(0)B_j+tsB_i^1(0)B_j^1(0)\right) \right],
\end{multline*}
for any $t,s \in [0,T]$.
\end{lemma}

\begin{proof}[Proof of Lemma \ref{LEM20sept2}]
Setting for every $i,j = 1, \ldots, k$, 
$$
\mathcal{B}_i(t) := B_i + t \, B_i^1(0) + \frac{t^2}{2} \, B_i^2(0) \qquad \forall t\in [0,T]
$$
 and using (\ref{ASSBB}), we check that for any $t,s \in [0,T]$,
$$
 \mathcal{B}_i(t) \mathcal{B}_j(s) =   \mathcal{P}_{i,j}(t,s) + \Delta_{i,j}(t,s),
 $$
 with
$$
\Delta_{i,j}(t,s) := \frac{t^2s}{2} \, B_i^2(0)B_j^1(0) + \frac{ts^2}{2} \, B_i^1(0)B_j^2(0) +  \frac{t^2s^2}{4} \, B_i^2(0)B_j^2(0).
$$
Moreover, remembering that 
$$
\frac{d^j}{dt^j} \left( S(t)^{-1}  B_i S(t)\right) = S(t)^{-1}  B_i^j(t) S(t) \qquad \forall t \in [0,T],
$$
we have
$$
S(t)^{-1}  B_i S(t) =\mathcal{B}_i(t) + O\bigl(t^3\bigr).
$$
Then by (\ref{30oct1}) we infer that for any $\delta \in (0,T)$ and any $u\in  L^2([0,\delta];\R^k)$, 
\begin{eqnarray*}
&\quad  & P \cdot D_0^2E^{I,T} (u) - Q_{\delta}(u)\\
& = & 2 \sum_{i,j=1}^k \int_0^{\delta} \int_0^t u_i(t) u_j(s)  \mbox{tr} \Bigl[ P^* S(T) \left(\mathcal{B}_i(t)+O\bigl(t^3\bigr)\right)\left(\mathcal{B}_j(s)+O\bigl(s^3\bigr)\right) \\
& \quad & \qquad \qquad \qquad \qquad \qquad \qquad \qquad \qquad \qquad \qquad -  \mathcal{P}_{i,j}(t,s)\Bigr]\,ds\,dt \\
& = & 2 \sum_{i,j=1}^k \int_0^{\delta} \int_0^t u_i(t) u_j(s) \mbox{tr} \Bigl[ P^* S(T) \Bigl( O\bigl(t^3\bigr) B_j(s) + B_i(t) O\bigl(s^3\bigr) +O\bigl(t^3\bigr)O\bigl(s^3\bigr) \\
& \quad & \qquad \qquad \qquad \qquad \qquad \qquad \qquad \qquad \qquad \qquad  +\Delta_{i,j}(t,s) \Bigr) \Bigr]\, ds \, dt.
\end{eqnarray*}
But for every nonnegative integers $p,q$ with $p+q\geq3$, we have
\begin{eqnarray*}
& \quad & \left|\sum_{i,j=1}^k \int_0^{\delta} \int_0^t u_i(t) u_j(s) t^p s^q \,ds\,dt \right| \\
& = & \left| \int_0^{\delta} \left(\sum_{i=1}^k u_i(t)t^p\right) \, \left(\int_0^t \sum_{j=1}^k u_j(s)) s^q\,ds\right) \,dt\right| \\
& \leq & \int_0^{\delta} \left(\sum_{i=1}^k\left|u_i(t)\right| t^p \right) \,  \left(\int_0^t \sum_{j=1}^k\left|u_j(s)\right|s^q \,ds\right)\,dt \\
& \leq & \int_0^{\delta} \left(\sum_{i=1}^k\left|u_i(t)\right| t^{p+q} \right) \,  \left(\int_0^t \sum_{j=1}^k\left|u_j(s)\right| \,ds\right)\,dt,
\end{eqnarray*}
which by Cauchy-Schwarz inequality yields
\begin{eqnarray*}
& \quad & \left|\sum_{i,j=1}^k \int_0^{\delta} \int_0^t u_i(t) u_j(s) t^p s^q \,ds\,dt \right| \\
& \leq &  \sqrt{\int_0^{\delta} \left(\sum_{i=1}^k\left|u_i(t)\right| t^{p+q} \right)^2\,dt} \, \sqrt{\int_0^{\delta} \left(\int_0^t \sum_{j=1}^k\left|u_j(s)\right| \,ds\right)^2\,dt} \\
& \leq & \sqrt{ k \int_0^{\delta} t^{2(p+q)} \, \sum_{i=1}^k\left|u_i(t)\right|^2\,dt} \,   \sqrt{\int_0^{\delta} t \, \int_0^t \left(\sum_{j=1}^k\left|u_j(s)\right| \right)^2 \, ds \,dt} \\
& \leq & \sqrt{ k    \delta^{2(p+q)}    \int_0^{\delta}  \sum_{i=1}^k\left|u_i(t)\right|^2\,dt} \,   \sqrt{\int_0^{\delta} t \, \int_0^{\delta} \left(\sum_{j=1}^k\left|u_j(s)\right| \right)^2 \, ds \,dt} \\
& \leq & \sqrt{k} \, \delta^3 \|u\|_{L^2} \,   \sqrt{ k  \|u\|_{L^2}^2 \int_0^{\delta} t  \,dt} = \frac{k}{\sqrt{2}} \, \delta^4 \|u\|_{L^2}^2. 
\end{eqnarray*}
We conclude easily.
\end{proof}

Returning to the proof of Proposition \ref{LIEPROP3}, we now want to show that the assumption (\ref{ind}) of Theorems \ref{THMopen}-\ref{THMopenquant} is satisfied. We are indeed going to show that a stronger property holds, namely that 
the index of the quadratic form in (\ref{ind}) goes to infinity as $\delta$ tends to zero. 

\begin{lemma}\label{LEM30oct}
For every integer $N>0$, there are $\delta>0$ and a subspace $L_{\delta} \subset L^2\bigl([0,\delta];\R^k\bigr)$ of dimension larger than $N$ such that the restriction of $Q_{\delta}$ to $L_{\delta}$ satisfies
$$
Q_{\delta} (u) \leq -2C \|u\|_{L^2}^2 \delta^4 \qquad \forall u \in L_{\delta}.
$$
\end{lemma}

\begin{proof}[Proof of Lemma \ref{LEM30oct}]
Using the notation 
$$
h_1 \odot h_2 = h_1(t) \odot h_2(s) :=   \int_0^{\delta} \int_0^t h_1(t) h_2(s)  \, ds \, dt,
$$
for any pair of continuous functions $h_1, h_2:[0,\delta] \rightarrow \R$, we check that for every $u\in L^2\bigl([0,\delta];\R^k\bigr)$,
\begin{eqnarray}\label{10dec1}
\frac{1}{2} \, Q_{\delta} (u) & = &   \sum_{i,j=1}^k \left( u_i \odot (su_j) \right) \, \mbox{tr} \Bigl[ P^* S(T) B_i B_j^1(0) \Bigr] \\
& \quad & \quad +   \sum_{i,j=1}^k  \left( (tu_i) \odot u_j \right) \, \mbox{tr} \Bigl[ P^* S(T) B_i^1(0) B_j \Bigr] \nonumber\\
& \quad & \quad + \sum_{i,j=1}^k   \left( u_i \odot \left(\frac{s^2u_j}{2}\right) \right) \,  \mbox{tr} \Bigl[ P^* S(T) B_i B_j^2(0) \Bigr]\nonumber \\ 
& \quad & \quad +  \sum_{i,j=1}^k   \left( \left(\frac{t^2u_i}{2}\right) \odot u_j \right) \,  \mbox{tr} \Bigl[ P^* S(T) B_i^2(0) B_j \Bigr] \nonumber\\ 
& \quad & \quad +  \sum_{i,j=1}^k \left( (tu_i) \odot (su_j) \right) \,   \mbox{tr} \Bigl[ P^* S(T) B_i^1(0) B_j^1(0) \Bigr].\nonumber
\end{eqnarray}
Fix $\bar{i}, \bar{j} \in \{1, \ldots, k\}$ with $\bar{i}\neq \bar{j}$ and take $v = \bigl( v_1, \ldots, v_k\bigr)\in L^2([0,\delta];\R^k)$ such that 
\begin{eqnarray*}
 v_i(t) =0 \quad \forall t \in [0,\delta], \, \forall  i  \in \{1, \ldots, k\} \setminus \{\bar{i},\bar{j}\}.
\end{eqnarray*}
The sum of the first two terms in the right-hand side of (\ref{10dec1}) is given by 
\begin{eqnarray*}
\sum_{i,j=1}^k \left\{ \left( v_{i} \odot (sv_{j}) \right) \, \mbox{tr} \Bigl[ P^* S(T) B_{i} B_{j}^1(0) \Bigr]+ \left( (tv_{i}) \odot v_{j} \right) \mbox{tr} \Bigl[ P^* S(T) B_{i}^1(0)B_{j} \Bigr] \right\} \\
= \left( v_{\bar{i}} \odot (sv_{\bar{j}}) \right) \, \mbox{tr} \Bigl[ P^* S(T) B_{\bar{i}} B_{\bar{j}}^1(0) \Bigr]
+\left( v_{\bar{j}} \odot (sv_{\bar{i}}) \right) \, \mbox{tr} \Bigl[ P^* S(T) B_{\bar{j}} B_{\bar{i}}^1(0) \Bigr]\\
+\left( v_{\bar{i}} \odot (sv_{\bar{i}}) \right) \, \mbox{tr} \Bigl[ P^* S(T) B_{\bar{i}} B_{\bar{i}}^1(0) \Bigr]
+\left( v_{\bar{j}} \odot (sv_{\bar{j}}) \right) \, \mbox{tr} \Bigl[ P^* S(T) B_{\bar{j}} B_{\bar{j}}^1(0) \Bigr]\\ 
+\left( (tv_{\bar{i}}) \odot v_{\bar{j}} \right) \, \mbox{tr} \Bigl[ P^* S(T) B_{\bar{i}}^1(0) B_{\bar{j}} \Bigr]
+\left( (tv_{\bar{j}}) \odot v_{\bar{i}} \right) \, \mbox{tr} \Bigl[ P^* S(T) B_{\bar{j}}^1(0) B_{\bar{i}} \Bigr]\\
+\left( (tv_{\bar{i}}) \odot v_{\bar{i}} \right) \, \mbox{tr} \Bigl[ P^* S(T) B_{\bar{i}}^1(0) B_{\bar{i}} \Bigr]
+\left( (tv_{\bar{j}}) \odot v_{\bar{j}} \right) \, \mbox{tr} \Bigl[ P^* S(T) B_{\bar{j}}^1(0) B_{\bar{j}} \Bigr].
\end{eqnarray*}
By integration by parts, we have 
$$
 v_{\bar{i}} \odot (sv_{\bar{i}})=\left(\int_0^{\delta} v_{\bar{i}}(s)\,ds \right) \left(\int_0^{\delta} sv_{\bar{i}}(s)\, ds\right)-(tv_{\bar{i}}) \odot v_{\bar{i}}).
$$
So
\begin{multline*}
\left( v_{\bar{i}} \odot (sv_{\bar{i}}) \right) \, \mbox{tr} \Bigl[ P^* S(T) B_{\bar{i}} B_{\bar{i}}^1(0) \Bigr]+\left( (tv_{\bar{i}}) \odot v_{\bar{i}} \right) \, \mbox{tr} \Bigl[ P^* S(T) B_{\bar{i}}^1(0) B_{\bar{i}} \Bigr]\\
=\left(\int_0^{\delta} v_{\bar{i}}(s)\, ds\right)\left(\int_0^{\delta} sv_{\bar{i}}(s)\, ds\right)\mbox{tr} \Bigl[ P^* S(T) B_{\bar{i}} B_{\bar{i}}^1(0) \Bigr] \\
+(tv_{\bar{i}}) \odot v_{\bar{i}})\mbox{tr} \Bigl[ P^* S(T) \left[B_{\bar{i}}^1(0), B_{\bar{i}}\right] \Bigr].
\end{multline*}
But according to (\ref{conditionLIEPROP4}) with $i=\bar{i}$ (remember that $\bar{t}=0$), we have
$$
\left[B_{\bar{i}}^1(0), B_{\bar{i}}\right]  \in \mbox{Span}  \Bigl\{B_r^s(0) \, \vert \, r =1,..,k, \, \, s \geq 0  \Bigr\},
$$
then by Lemma \ref{LEM20sept1} we obtain 
$$
\mbox{tr} \Bigl[ P^* S(T) \left[B_{\bar{i}}^1(0), B_{\bar{i}}\right] \Bigr]=0,
$$
and consequently,
\begin{multline*}
\left( v_{\bar{i}} \odot (sv_{\bar{i}}) \right) \, \mbox{tr} \Bigl[ P^* S(T) B_{\bar{i}} B_{\bar{i}}^1(0) \Bigr]+\left( (tv_{\bar{i}}) \odot v_{\bar{i}} \right) \, \mbox{tr} \Bigl[ P^* S(T) B_{\bar{i}}^1(0) B_{\bar{i}} \Bigr]\\
=\left(\int_0^{\delta} v_{\bar{i}}(s)\,ds\right)\left(\int_0^{\delta} sv_{\bar{i}}(s)\,ds\right)\mbox{tr} \Bigl[ P^* S(T) B_{\bar{i}} B_{\bar{i}}^1(0) \Bigr].
\end{multline*}
Similarly, we obtain 
\begin{eqnarray*}
\left( v_{\bar{j}} \odot (sv_{\bar{j}}) \right) \, \mbox{tr} \Bigl[ P^* S(T) B_{\bar{j}} B_{\bar{j}}^1(0) \Bigr]+\left( (tv_{\bar{j}}) \odot v_{\bar{j}} \right) \, \mbox{tr} \Bigl[ P^* S(T) B_{\bar{j}}^1(0) B_{\bar{j}} \Bigr]\\
=\left(\int_0^{\delta} v_{\bar{j}}(s)\,ds\right)\left(\int_0^{\delta} sv_{\bar{j}}(s)\, ds\right)\mbox{tr} \Bigl[ P^* S(T) B_{\bar{j}} B_{\bar{j}}^1(0) \Bigr].
\end{eqnarray*}
In conclusion, the sum of the first two terms in the right-hand side of (\ref{10dec1}) can be written as 
\begin{multline*}
\sum_{i,j=1}^k \left\{ \left( v_{i} \odot (sv_{j}) \right) \, \mbox{tr} \Bigl[ P^* S(T) B_{i} B_{j}^1(0) \Bigr]+ \left( (tv_{i}) \odot v_{j} \right) \mbox{tr} \Bigl[ P^* S(T) B_{i}^1(0)B_{j} \Bigr] \right\} \\
= \left( v_{\bar{i}} \odot (sv_{\bar{j}}) \right) \, \mbox{tr} \Bigl[ P^* S(T) B_{\bar{i}} B_{\bar{j}}^1(0) \Bigr]
+\left( v_{\bar{j}} \odot (sv_{\bar{i}}) \right) \, \mbox{tr} \Bigl[ P^* S(T) B_{\bar{j}} B_{\bar{i}}^1(0) \Bigr]\\
+\left(\int_0^{\delta} v_{\bar{i}}(s)\, ds\right)\left(\int_0^{\delta} sv_{\bar{i}}(s)\, ds\right)\mbox{tr} \Bigl[ P^* S(T) B_{\bar{i}} B_{\bar{i}}^1(0) \Bigr]\\
+\left( (tv_{\bar{i}}) \odot v_{\bar{j}} \right) \, \mbox{tr} \Bigl[ P^* S(T) B_{\bar{i}}^1(0) B_{\bar{j}} \Bigr]
+\left( (tv_{\bar{j}}) \odot v_{\bar{i}} \right) \, \mbox{tr} \Bigl[ P^* S(T) B_{\bar{j}}^1(0) B_{\bar{i}} \Bigr]\\
+\left(\int_0^{\delta} v_{\bar{j}}(s)\, ds\right)\left(\int_0^{\delta} sv_{\bar{j}}(s)\, ds\right)\mbox{tr} \Bigl[ P^* S(T) B_{\bar{j}} B_{\bar{j}}^1(0) \Bigr].
\end{multline*}
By the same arguments as above, the sum of the third and fourth terms in  the right-hand side of (\ref{10dec1}) can be written as
\begin{multline*}
\sum_{i,j=1}^k \left\{ \left( v_{i} \odot \left(\frac{s^2v_{j}}{2}\right) \right) \, \mbox{tr} \Bigl[ P^* S(T) B_{i} B_{j}^2(0) \Bigr] \right. \\
\left. + \left( \left(\frac{t^2v_{i}}{2}\right) \odot v_{j} \right) \mbox{tr} \Bigl[ P^* S(T) B_{i}^2(0)B_{j} \Bigr] \right\} \\
= \left( v_{\bar{i}} \odot \left(\frac{s^2v_{\bar{j}}}{2}\right) \right) \, \mbox{tr} \Bigl[ P^* S(T) B_{\bar{i}} B_{\bar{j}}^2(0) \Bigr]+\left( v_{\bar{j}} \odot \left(\frac{s^2v_{\bar{i}}}{2} \right) \right) \, \mbox{tr} \Bigl[ P^* S(T) B_{\bar{j}} B_{\bar{i}}^2(0) \Bigr]\\
+\left(\int_0^{\delta} v_{\bar{i}}(s)\, ds\right)\left(\int_0^{\delta} \frac{s^2v_{\bar{i}}(s) }{2}\, ds\right)\mbox{tr} \Bigl[ P^* S(T) B_{\bar{i}} B_{\bar{i}}^2(0) \Bigr]\\
+\left( \left(\frac{t^2v_{\bar{i}}}{2}\right) \odot v_{\bar{j}} \right) \, \mbox{tr} \Bigl[ P^* S(T) B_{\bar{i}}^2(0) B_{\bar{j}} \Bigr]+\left( \left(\frac{t^2v_{\bar{j}}}{2}\right) \odot v_{\bar{i}} \right) \, \mbox{tr} \Bigl[ P^* S(T) B_{\bar{j}}^2(0) B_{\bar{i}} \Bigr]\\
+\left(\int_0^{\delta} v_{\bar{j}}(s)\, ds\right)\left(\int_0^{\delta} \frac{s^2v_{\bar{j}}(s)}{2} \,ds\right)\mbox{tr} \Bigl[ P^* S(T) B_{\bar{j}} B_{\bar{j}}^2(0) \Bigr],
\end{multline*}
the fifth (and last) part of $\frac{1}{2} \, Q_{\delta} (v)$ is given by
\begin{multline*}
\sum_{i,j=1}^k \left\{ \left( (tv_{i}) \odot (sv_{j}) \right) \, \mbox{tr} \Bigl[ P^* S(T) B_{i}^1(0)B_{j}^1(0) \Bigr] \right\} = \\
\left( (tv_{\bar{i}}) \odot (sv_{\bar{j}}) \right) \, \mbox{tr} \Bigl[ P^* S(T) B_{\bar{i}}^1(0) B_{\bar{j}}^1(0) \Bigr]+ \left( (tv_{\bar{j}}) \odot (sv_{\bar{i}}) \right) \, \mbox{tr} \Bigl[ P^* S(T) B_{\bar{j}}^1(0) B_{\bar{i}}^1(0) \Bigr]\\
 +\left( (tv_{\bar{i}}) \odot (sv_{\bar{i}}) \right) \, \mbox{tr} \Bigl[ P^* S(T) (B_{\bar{i}}^1(0))^2 \Bigr]\\
+\left( (tv_{\bar{j}}) \odot v_{\bar{j}}) \right) \, \mbox{tr} \Bigl[ P^* S(T) (B_{\bar{j}}^1(0))^2 \Bigr].
\end{multline*}
By integration by parts, we have 
$$
(tv_{\bar{i}}) \odot (sv_{\bar{i}})=\frac{1}{2}\left(\int_0^{\delta} sv_{\bar{i}}(s) \, ds \right)^2, \quad (tv_{\bar{j}}) \odot (sv_{\bar{j}})=\frac{1}{2}\left(\int_0^{\delta} sv_{\bar{j}}(s) \, ds \right)^2,
$$
$$
\mbox{and} \quad (tv_{\bar{j}}) \odot (sv_{\bar{i}}) = \left(\int_0^{\delta} sv_{\bar{i}}(s) \, ds \right) \left(\int_0^{\delta} sv_{\bar{j}}(s) \, ds \right) -  (tv_{\bar{i}}) \odot (sv_{\bar{j}}).
$$
Therefore the last part of $\frac{1}{2} \, Q_{\delta} (v)$ can be written as 
\begin{multline*}
\sum_{i,j=1}^k \left\{ \left( (tv_{i}) \odot (sv_{j}) \right) \, \mbox{tr} \Bigl[ P^* S(T) B_{i}^1(0)B_{j}^1(0) \Bigr] \right\} = \\
 \left( (tv_{\bar{i}}) \odot (sv_{\bar{j}}) \right) \, \mbox{tr} \Bigl[ P^* S(T) \left[ B_{\bar{i}}^1(0), B_{\bar{j}}^1(0)\right] \Bigr] \\
 +    \left(\int_0^{\delta} sv_{\bar{i}}(s) \, ds \right) \left(\int_0^{\delta} sv_{\bar{j}}(s) \, ds \right)       \, \mbox{tr} \Bigl[ P^* S(T) B_{\bar{j}}^1(0) B_{\bar{i}}^1(0) \Bigr]\\
+\frac{1}{2}\left(\int_0^{\delta} sv_{\bar{i}}(s) \, ds \right)^2 \, \mbox{tr} \Bigl[ P^* S(T) (B_{\bar{i}}^1(0))^2 \Bigr]\\
 + \frac{1}{2}\left(\int_0^{\delta} sv_{\bar{j}}(s) \, ds \right)^2 \, \mbox{tr} \Bigl[ P^* S(T) (B_{\bar{j}}^1(0))^2 \Bigr].
\end{multline*}
To summarize, we have
\begin{multline*}
\frac{1}{2} \, Q_{\delta} (v)  =  \\
  \left( v_{\bar{i}} \odot (sv_{\bar{j}}) \right) \, \mbox{tr} \Bigl[ P^* S(T) B_{\bar{i}} B_{\bar{j}}^1(0) \Bigr] +\left( v_{\bar{j}} \odot (sv_{\bar{i}}) \right) \, \mbox{tr} \Bigl[ P^* S(T) B_{\bar{j}} B_{\bar{i}}^1(0) \Bigr]\\
 +\left(\int_0^{\delta} v_{\bar{i}}(s)\, ds\right)\left(\int_0^{\delta} sv_{\bar{i}}(s)\, ds\right)\mbox{tr} \Bigl[ P^* S(T) B_{\bar{i}} B_{\bar{i}}^1(0) \Bigr]\\
 +\left( (tv_{\bar{i}}) \odot v_{\bar{j}} \right) \, \mbox{tr} \Bigl[ P^* S(T) B_{\bar{i}}^1(0) B_{\bar{j}} \Bigr]+\left( (tv_{\bar{j}}) \odot v_{\bar{i}} \right) \, \mbox{tr} \Bigl[ P^* S(T) B_{\bar{j}}^1(0) B_{\bar{i}} \Bigr]\\
+\left(\int_0^{\delta} v_{\bar{j}}(s)\, ds\right)\left(\int_0^{\delta} sv_{\bar{j}}(s)\, ds\right)\mbox{tr} \Bigl[ P^* S(T) B_{\bar{j}} B_{\bar{j}}^1(0) \Bigr]\\
 + \left( v_{\bar{i}} \odot \left(\frac{s^2v_{\bar{j}}}{2}\right) \right) \, \mbox{tr} \Bigl[ P^* S(T) B_{\bar{i}} B_{\bar{j}}^2(0) \Bigr]+\left( v_{\bar{j}} \odot \left(\frac{s^2v_{\bar{i}}}{2}\right) \right) \, \mbox{tr} \Bigl[ P^* S(T) B_{\bar{j}} B_{\bar{i}}^2(0) \Bigr]\\
 +\left(\int_0^{\delta} v_{\bar{i}}(s)\, ds\right)\left(\int_0^{\delta} \frac{s^2v_{\bar{i}}(s)}{2} \, ds\right)\mbox{tr} \Bigl[ P^* S(T) B_{\bar{i}} B_{\bar{i}}^2(0) \Bigr]\\
+\left( \left(\frac{t^2v_{\bar{i}}}{2}\right) \odot v_{\bar{j}} \right) \, \mbox{tr} \Bigl[ P^* S(T) B_{\bar{i}}^2(0) B_{\bar{j}} \Bigr] + \left( \left(\frac{t^2v_{\bar{j}}}{2}\right) \odot v_{\bar{i}} \right) \, \mbox{tr} \Bigl[ P^* S(T) B_{\bar{j}}^2(0) B_{\bar{i}} \Bigr]\\
 +\left(\int_0^{\delta} v_{\bar{j}}(s)\, ds\right)\left(\int_0^{\delta} \frac{s^2v_{\bar{j}}(s)}{2} \, ds\right)\mbox{tr} \Bigl[ P^* S(T) B_{\bar{j}} B_{\bar{j}}^2(0) \Bigr]\\
 +    \left(\int_0^{\delta} sv_{\bar{i}}(s) \, ds \right) \left(\int_0^{\delta} sv_{\bar{j}}(s) \, ds \right)       \, \mbox{tr} \Bigl[ P^* S(T) B_{\bar{j}}^1(0) B_{\bar{i}}^1(0) \Bigr]\\
 +\frac{1}{2}\left(\int_0^1 sv_{\bar{i}}(s) \, ds \right)^2 \, \mbox{tr} \Bigl[ P^* S(T) (B_{\bar{i}}^1(0))^2 \Bigr] \\
 + \frac{1}{2}\left(\int_0^1 sv_{\bar{j}}(s) \, ds \right)^2 \, \mbox{tr} \Bigl[ P^* S(T) (B_{\bar{j}}^1(0))^2 \Bigr]\\
 +  \left( (tv_{\bar{i}}) \odot (sv_{\bar{j}}) \right) \, \mbox{tr} \Bigl[ P^* S(T) \left[ B_{\bar{i}}^1(0), B_{\bar{j}}^1(0)\right] \Bigr].
\end{multline*}
We now need the following technical result whose proof is given in Appendix.

\begin{lemma}\label{LEMtechnical1nov}
Denote by $\mathcal{L}_{\bar{i},\bar{j}}$ the set of 
$$
v = \bigl( v_1, \ldots, v_k\bigr)\in L^2([0,1];\R^k)
$$
such that 
\begin{eqnarray*}
 v_i(t) =0 \quad \forall t \in [0,1], \, \forall  i  \in \{1, \ldots, k\} \setminus \{\bar{i},\bar{j}\},
\end{eqnarray*}
\begin{eqnarray*}
\int_0^1 v_{\bar{i}}(s)\, ds = \int_0^1 s v_{\bar{i}}(s) \,ds = \int_0^1 v_{\bar{j}}(s) \,ds = \int_0^1 s v_{\bar{j}}(s)\, ds = 0 ,
\end{eqnarray*} 
\begin{eqnarray*}
v_{\bar{i}} \odot (sv_{\bar{j}}) = v_{\bar{j}} \odot (sv_{\bar{i}}) = v_{\bar{i}} \odot (s^2v_{\bar{j}}) = v_{\bar{j}} \odot (s^2v_{\bar{i}}) = 0,
\end{eqnarray*}
and 
\begin{eqnarray*}
(tv_{\bar{i}}) \odot (sv_{\bar{j}}) > 0.
\end{eqnarray*}
Then, for every integer $N>0$, there are  a vector space  $L_{\bar{i},\bar{j}}^N \subset \mathcal{L}_{\bar{i},\bar{j}} \cup \{0\}$ of dimension $N$ and a constant $K(N)>0$  such that
\begin{eqnarray*}
 \bigl(tv_{\bar{i}}\bigr) \odot \bigl(sv_{\bar{j}}\bigr) \geq \frac{1}{K(N)} \|v\|_{L^2}^2 \qquad \forall v \in L_{\bar{i},\bar{j}}^N.
\end{eqnarray*}
\end{lemma}

Let us now show how to conclude the proof of Lemma \ref{LEM30oct}. Recall that  $P \in T_{\bar{X}(T)} \mbox{Sp}(m)$ was fixed such that $P$ belongs to $\left( \mbox{Im} \bigl( D_{0} E^{I,T} \bigr)\right)^{\perp} \setminus \{0\}$ and that by Lemma 
\ref{LEM20sept1}, we know that (taking $t=0$)
$$
P \cdot S(T)B_i^j(0) = 0 \qquad \forall j\geq 0, \, \forall i \in {1,..,k}.
$$
By (\ref{conditionLIEPROP3}) ($\bar{t}=0$), we also have 
\begin{multline*}
 \mbox{Span} \Bigl\{S(T)B_i^j(0),S(T)[B_i^1(0),B_s^1(0)] \, \vert \, i,s \in {1,..,k}, \, j=0,1,2 \Bigr\}= T_{\bar{X}(T)}Sp(m).
\end{multline*}
Consequently, we infer that there are $\bar{i},\bar{j}  \in \{1, \ldots, k\}$ with $\bar{i}\neq \bar{j}$ such that
\begin{eqnarray*}
  \mbox{tr} \Bigl( P^* S(T) \left[B_{\bar{i}}^1(0), B_{\bar{j}}^1(0)\right] \Bigr) <0.
\end{eqnarray*}
Let $N>0$ an integer be fixed, $L_{\bar{i},\bar{j}}^N \subset \mathcal{L}_{\bar{i},\bar{j}} \cup \{0\}$ of dimension $N$ and the constant $K(N)>0$ given by Lemma \ref{LEMtechnical1nov}, for every $\delta\in (0,t)$ denote by $L_{\delta}^N$ the vector space of $u\in  L^2\bigl([0,\delta];\R^k\bigr) \subset L^2\bigl([0,T];\R^k\bigr)$ such that there is $v\in L_{\bar{i},\bar{j}}$ satisfying
$$
u(t) = v(t/\delta) \qquad \forall t \in [0,\delta].
$$
For every $v  \in L_{\bar{i},\bar{j}}$, the control $u_{\delta}:[0,T]\rightarrow \R^k$ defined by 
$$
 u_{\delta}(t):= v(t/\delta) \qquad  t \in  [0,\delta]
$$
belongs to $L_{\delta}^N$ and by an easy change of variables,
$$
\bigl\|u_{\delta}\bigr\|^2 = \int_0^T \bigl| u_{\delta}(t)\bigr|^2 \, dt = \int_0^{\delta} \bigl| u_{\delta}(t)\bigr|^2 \, dt = \delta \int_0^{1} | v(t)|^2 \, dt = \delta \|v\|^2.
$$
Moreover it satisfies
$$
Q_{\delta} (u_{\delta}) = 2 \left(  \bigl(tv_{\bar{i}}\bigr) \odot \bigl(sv_{\bar{j}}\bigr)\right) \,  \delta^4 \,  \mbox{tr} \Bigl( P^* S(T) \left[B_{\bar{i}}^1(0), B_{\bar{j}}^1(0)\right] \Bigr).
$$
Then we infer that 
\begin{eqnarray*}
 \frac{Q_{\delta}(u_{\delta})}{\|u_{\delta}\|_{L^2}^2 \delta^4} &= & \frac{2 \left(  \bigl(tv_{\bar{i}}\bigr) \odot \bigl(sv_{\bar{j}}\bigr)\right) }{\delta \|v\|_{L^2}^2}  
\,  \mbox{tr} \Bigl( P^* S(T) \left[B_{\bar{i}}^1(0), B_{\bar{j}}^1(0)\right] \Bigr)\\
&  \leq &  \frac{2}{\delta K(N)} \,  \mbox{tr} \Bigl( P^* S(T) \left[B_{\bar{i}}^1(0), B_{\bar{j}}^1(0)\right] \Bigr).
\end{eqnarray*}
We get the result for $\delta>0$ small enough.
\end{proof}

We can now conclude the proof of Proposition  \ref{LIEPROP3}.  First we note that given $N\in \N$ strictly larger than $m(2m+1)$, if $L \subset L^2\bigl( [0,T];\R^k\bigr)$ is a vector space of dimension $N$, then the linear operator
$$
\left( D_0E^{I,T}\right)_{\vert L} \,: \, L \rightarrow T_{\bar{X}(T)} \mbox{Sp}(m) \subset M_{2m}(\R)
$$
has a kernel of dimension at least $N-m(2m+1)$, which means that 
$$
\mbox{Ker} \left( D_0E^{I,T}\right) \cap L
$$
has dimension at least $N-m(2m+1)$. Then, thanks to Lemma \ref{LEM30oct}, for every integer $N>0$, there are $\delta > 0$ and a subspace $L_{\delta}\subset  L^2\bigl([0,\delta];\R^k\bigr) \subset L^2\bigl([0,T];\R^k\bigr)$ such that the dimension of $\tilde{L}_{\delta}:=L_{\delta} \cap  \mbox{Ker} \left( D_0E^{I,T}\right) $ is larger than $N$ and the restriction of $Q_{\delta}$ to $\tilde{L}_{\delta}$ satisfies
$$
Q_{\delta} (u) \leq  -2C \|u\|_{L^2}^2 \delta^4 \qquad \forall u \in \tilde{L}_{\delta}.
$$
By Lemma \ref{LEM20sept2}, we have 
$$
 P \cdot D_0^2E^{I,T} (u)  \leq   Q_{\delta}(u)+C\delta^4 \, \|u\|_{L^2}^2 \qquad \forall u \in \tilde{L}_{\delta}.
$$
Then we infer that
\begin{eqnarray}\label{P7fev1}
 P \cdot D_0^2E^{I,T} (u) \leq - C\delta^4 \, \|u\|_{L^2}^2  < 0 \qquad \forall u \in \tilde{L}_{\delta}.
\end{eqnarray}
Note that since $E^{I,T}$ is valued in $\mbox{Sp}(m)$ which is a submanifold of $M_{2m}(\R)$, assumption (\ref{ind}) is not satisfied and Theorems \ref{THMopen} and \ref{THMopenquant} do not apply. 

Let $\Pi : M_{2m}(\R) \rightarrow T_{\bar{X}(T)}Sp(m)$ be the orthogonal projection onto $T_{\bar{X}(T)}Sp(m)$. Its restriction to $\mbox{Sp}(m)$, $\bar{\Pi}:=\Pi_{\vert \mbox{Sp}(m)}$, is a smooth mapping whose differential at $\bar{X}(T)$ 
is equal to the identity of $T_{\bar{X}(T)}Sp(m)$ so it is an isomorphism. Thanks to the Inverse Function Theorem (for submanifolds), $\bar{\Pi}$ is a local $C^{\infty}$-diffeomorphism at $\bar{X}(T)$. Hence there exist $\mu > 0$ such that the restriction of $\bar{\Pi}$ to $B \Bigl( \bar{X}(T) ,\mu \Bigr) \cap \mbox{Sp}(m)$ 
$$
\bar{\Pi}_{\vert B( \bar{X}(T) ,\mu) \cap Sp(m)} : B \Bigl( \bar{X}(T) ,\mu \Bigr) \cap \mbox{Sp}(m) \rightarrow \bar{\Pi}\Bigl(B \Bigl( \bar{X}(T) ,\mu \Bigr) \cap \mbox{Sp}(m)\Bigr)
$$ 
is a smooth diffeomorphism. The map $E^{I,T}$ is continuous so $$
\mathcal{U}:=(E^{I,T})^{-1}\Bigl(B \Bigl( \bar{X}(T) ,\mu \Bigr) \cap \mbox{Sp}(m)\Bigr)
$$
 is an open set of $L^2([0,T];\R^k)$ containing $\bar{u}=0$. Define the function $F : \mathcal{U} \rightarrow 
T_{\bar{X}(T)}Sp(m)$ by $F:=\bar{\Pi} \circ E^{I,T}=\Pi \circ E^{I,T}$. The mapping $F$ is  $C^{2}$ and we have
$$
F(\bar{u}) = \bar{X}(T), \quad D_{\bar{u}}F=D_{\bar{u}}E^{I,T} \quad \mbox{and} \quad D^2_{\bar{u}}F=\Pi \circ D^2_{\bar{u}}E^{I,T}.
$$
Let us check that $F$ satisfies assumption (\ref{ind}). For every $P \in T_{\bar{X}(T)} \mbox{Sp}(m)$ such that $P$ belongs to $\left( \mbox{Im} \bigl( D_{\bar{u}} F \bigr)\right)^{\perp} \setminus \{0\}$ and every $v \in L^2 ([0,T];\R^k)$, we have
$$
P \cdot D_{\bar{u}}^2E^{I,T} (v)=P \cdot \Pi \circ D_{\bar{u}}^2E^{I,T} (u) + P \cdot \Bigl( D_{\bar{u}}^2E^{I,T} (u)-\Pi \circ D_{\bar{u}}^2E^{I,T} (u)\Bigr).
$$
But
$$
D_{\bar{u}}^2E^{I,T} (u)-\Pi \circ D_{\bar{u}}^2E^{I,T} (u) \in \Bigl(T_{\bar{X}(T)}\mbox{Sp}(m)\Bigr)^{\perp}, 
$$
hence
$$
P \cdot D_{\bar{u}}^2E^{I,T} (u)=P \cdot D^2_{\bar{u}}F (u).
$$ 
Therefore, by (\ref{P7fev1}), assumption (\ref{ind}) is satisfied. Consequently, thanks to Theorem \ref{THMopenquant} there exist $\bar{\epsilon}, c  \in (0,1)$ such that for every $\epsilon \in (0,\bar{\epsilon})$ the following property holds: For every $u \in \mathcal{U}, Z \in T_{\bar{X}(T)}\mbox{Sp}(m)$ with
\begin{eqnarray*}
\left\| u - \bar{u} \right\|_{L^2} < \epsilon, \quad \left| Z- F(u) \right| < c\, \epsilon^2,
\end{eqnarray*}
there are $w_1, w_2 \in L^2\bigl([0,T];\R^k\bigr)$ such that $u+w_1+w_2\in \mathcal{U}$, 
\begin{eqnarray*}
Z = F \bigl(u+ w_1 + w_2\bigr),
\end{eqnarray*}
and
\begin{eqnarray*}
w_1 \in  \mbox{Ker} \left(D_{u} F\right), \quad \bigl\| w_1\bigr\|_{L^2} < \epsilon, \quad \bigl\| w_2\bigr\|_{L^2} < \epsilon^2.
\end{eqnarray*} 
Apply the above property with $u=\bar{u}$ and $X\in \mbox{Sp}(m)$ such that 
$$
\bigl| X- \bar{X}(T) \bigr| =: \frac{ c\epsilon^2}{2} \mbox{ with } \epsilon < \bar{\epsilon}.
$$
Set $Z:= \Pi(X)$, then we have ($\Pi$ is an orthogonal projection so it is 1-lipschitz)
$$
\left| Z- F(\bar{u}) \right|=\left| \Pi(X)- \Pi(\bar{X}(T)) \right| \leq \left| X- \bar{X}(T) \right| = \frac{ c\epsilon^2}{2} < c \epsilon^2.
$$
Therefore by the above property, there are $w_1, w_2 \in L^2\bigl([0,T];\R^k\bigr)$ such that $\tilde{u}:=\bar{u}+w_1+w_2\in \mathcal{U}$ satisfies 
\begin{eqnarray*}
Z = F \bigl(\tilde{u}\bigr),
\end{eqnarray*}
and
\begin{eqnarray*}
\bigl\| \tilde{u} \bigr\|_{L^2} \leq \left\| w_1\right\|_{L^2} + \left\| w_2 \right\|_{L^2} \leq \epsilon + \epsilon^2.
\end{eqnarray*}
Since $\bar{\Pi}_{\vert B( \bar{X}(T) ,\mu) \cap Sp(m)}$ is a local diffeomorphism, taking $\epsilon >0$ small enough, we infer that 
$$
X = E^{I,T} \bigl(\tilde{u}\bigr) \quad \mbox{and} \quad \bigl\|\tilde{u}\bigr\|_{L^2} \leq 2 \epsilon = 2 \sqrt{\frac{2}{c}}\, \left| X- \bar{X}(T) \right|^{1/2}.
$$
In conclusion, the control system (\ref{syscontrol}) is controllable at second order around $\bar{u} \equiv 0$, which concludes the proof of Proposition 2.2.

\subsection{Proof of Proposition \ref{LIEPROP3bis}}\label{ProofLIEPROP3bis}
As in the proof of Proposition \ref{LIEPROP3}, we may assume without loss of generality that $\bar{X}=I_{2m}$. Recall that for every $\theta \in \Theta$, $E_{\theta}^{I,T}  :  L^2 \bigl([0,T]; \R^k\bigr)   \rightarrow \mbox{Sp}(m) \subset M_{2m}(\R)$ denotes the End-Point mapping associated with  (\ref{syscontrol}) with parameter $\theta$ starting at $I=I_{2m}$. Given $\theta \in \Theta$ two cases may appear, either $E_{\theta}^{I,T}$ is a submersion at $\bar{u}\equiv 0$ or is not submersion at $\bar{u}\equiv 0$. Let us denote by $\Theta_1 \subset \Theta$ the set of parameters $\theta$ where  $E_{\theta}^{I,T}$ is submersion at $\bar{u}\equiv 0$ and by $\Theta_2$ its complement in $\Theta$. By continuity of the mapping $\theta \mapsto  D_0E_{\theta}^{I,T}$ the set $\Theta_1$ is open in $\Theta$  while $\Theta_2$ is compact.

For every $\theta \in \Theta_1$, since $E_{\theta}^{I,T}$ is submersion at $\bar{u}$, we have uniform controllability at first order around $\bar{u}$ for a set of parameters close to $\bar{\theta}$.  So we need to show that we have controllability at second order around $\bar{u}$ for any parameter in some neighborhood of $\Theta_2$. 

By the proof of Proposition  \ref{LIEPROP3} (see (\ref{P7fev1})), for every $\theta \in \Theta_2$, every  $P$ in the nonempty set  $\left( \mbox{Im} \bigl( D_{0} E_{\theta}^{I,T} \bigr)\right)^{\perp} \setminus \{0\}$ and every integer $N>0$ there exists a finite dimensional subspace $L_{\theta,P,N}  \subset L^2([0,T];\R^k)$ with 
$$
D:= \dim \left(  L_{\theta,P,N} \right) >N,
$$
 such that 
$$
 P \cdot D_0^2E^{I,T}_{\theta} (u) < 0 \qquad \forall u \in L_{\theta,P,N} \setminus \{0\}
$$
and 
$$
 \dim \left(  L_{\theta,P,N}\cap  \mbox{Ker} \left( D_0E_{\theta}^{I,T} \right) \right)\geq N-m(2m+1).
$$
By bilinearity of $u \mapsto  P \cdot D_0^2E^{I,T}_{\theta} (u)$ and compactness of the sphere in $ L_{\theta,P,N}$, there is $C_{\theta,P,N}>0$ such that 
 $$
 P \cdot D_0^2E_{\theta}^{I,T} (u)  \leq  -C_{\theta,P,N} \, \|u\|_{L^2}^2 \qquad \forall u \in L_{\theta,P,N}.
$$
Let $u^1, \ldots, u^D \in L^2([0,T];\R^k)$ be a basis of  $L_{\theta,P,N}$ such that 
$$
\|u^i\|_{L^2} = 1 \qquad \forall i=1, \ldots, D.
$$
Since the set of controls $u \in C^{\infty}([0,T],\R^k)$ with $\mbox{Supp} (u) \subset (0,T)$  is dense in $L^2([0,T],\R^k)$, there is a linearly independent family $\tilde{u}^1, \ldots, \tilde{u}^D$ in $C^{\infty}([0,T],\R^k)$ with $\mbox{Supp} (u) \subset (0,T)$ (from now we will denote by  $C_0^{\infty}([0,T],\R^k)$ the set of functions in  $C^{\infty}([0,T],\R^k)$  with support in $(0,T)$) such that 
 $$
 P \cdot D_0^2E_{\theta}^{I,T} (u)  \leq  - \frac{C_{\theta,P,N}}{2} \, \|u\|_{L^2}^2 \qquad \forall u \in \tilde{L}_{\theta,P,N}:=\mbox{Span} \left\{\tilde{u}^i \, \vert \, i=1, \ldots, D\right\}.
$$
Moreover by continuity of the mapping $(P,\theta)\mapsto  P \cdot D_0^2E_{\theta}^{I,T}$, we may also assume that the above inequality holds for any $\tilde{\theta}$ close to $\theta$ and $\tilde{P}$ close to $P$.  Let an integer $N>0$ be fixed, we check easily that the set 
$$
\mathcal{A} := \Bigl\{ (\theta,P) \in \Theta \times M_{2m}(\R) \, \vert \, \|P\|=1, \, P \in     \left( \mbox{Im} \bigl( D_{0} E_{\theta}^{I,T} \bigr)\right)^{\perp} \Bigr\}
$$
is compact. Therefore, by the above discussion there is a finite family $\{(\theta_a,P_a)\}_{a=1, \ldots, A}$ in $\mathcal{A}$ together with a finite family of open neighborhoods  $\{\mathcal{V}_a\}_{a=1, \ldots, A}$ of the pairs $(\theta_a,P_a)$ ($a=1, \ldots, A$) in $\mathcal{A}$ such that 
$$
\mathcal{A}  = \bigcup_{a=1}^A \mathcal{V}_a
$$
and there is a finite family of $\{\tilde{L}_a\}_{a=1, \ldots, A}$ of finite dimensional subspaces in  $C^{\infty}_0([0,T],\R^k)$ such that 
 $$
 P \cdot D_0^2E_{\theta}^{I,T} (u)  <0 \qquad \forall u \in \tilde{L}_a \setminus \{0\},
$$
for every $a\in \{1, \ldots, A\}$ and any $(\theta,P)$ in $\mathcal{V}_a$. Then set 
$$
\tilde{L}(N) :=  \bigcup_{a=1}^A \tilde{L}_a    \, \subset C^{\infty}_0([0,T],\R^k),
$$
pick a basis $\tilde{u}^1, \ldots, \tilde{u}^B$ of $\tilde{L}(N)$ and define $F^N : \Theta \times \R^B \rightarrow \mbox{Sp}(m)$ by 
$$
F^N_{\theta}(\lambda):=E_{\theta}^{I,T}\left(\sum_{b=1}^B \lambda_b \tilde{u}^b\right) \qquad \forall \lambda=(\lambda_1,...,\lambda_B) \in \R^B, \, \forall \theta \in \Theta.
$$
By construction, $F^N$ is at least $C^2$ and for every $\theta\in \Theta_2$ and every $P \in \left( \mbox{Im} \bigl( D_{0} F_{\theta}^N \bigr)\right)^{\perp} \setminus \{0\}$, there is a subspace $L_{\theta,P}^N \subset \tilde{L}(N)$ such that
$$
\dim \left(  L_{\theta,P}^N \right) >N,
$$
$$
 P \cdot D_0^2F_{\theta}^N (u) < 0 \qquad \forall u \in L_{\theta,P}^N \setminus \{0\}
$$
and 
$$
 \dim \left(  L_{\theta,P}^N \cap  \mbox{Ker} \left( D_0F_{\theta}^{N} \right) \right)\geq N-m(2m+1).
$$
As in the proof of Proposition \ref{LIEPROP3}, we need to be careful because $F^N$ is valued in $\mbox{Sp}(m)$. Given $\bar{\theta}$, we denote by $\Pi_{\bar{\theta}} : M_{2m}(\R) \rightarrow T_{\bar{X}_{\bar{\theta}}(T)}Sp(m)$ the orthogonal projection onto $T_{\bar{X}_{\bar{\theta}}(T)}\mbox{Sp}(m)$ and observe that the restriction of $\Pi$ to $ T_{\bar{X}_{\theta}(T)}\mbox{Sp}(m)$  is an isomorphism for $\theta \in \mathcal{W}_{\bar{\theta}}$ an open neighborhood of  $\bar{\theta}$. Then we define  $G^{N,\bar{\theta}} : \Theta \times \R^B \rightarrow  T_{\bar{X}_{\bar{\theta}}(T)}\mbox{Sp}(m)$ by 
$$
G_{\theta}^{N,\bar{\theta}}(\lambda):=\Pi_{\bar{\theta}} \left(  F_{\theta}^N(\lambda) \right)  \qquad \forall \lambda \in \R^B, \, \forall \theta \in \mathcal{W}_{\bar{\theta}}.
$$
Taking $N$ large enough, by compactness of $\Theta_2$, a parametric version of  Theorem \ref{THMopenquant} (see \cite{lazragthesis}) yields  $\bar{\epsilon}, c  \in (0,1)$ such that for every $\epsilon \in (0,\bar{\epsilon})$ and for any $\bar{\theta} \in \Theta_2$ the following property holds:
For every $\theta \in \mathcal{W}_{\bar{\theta}}, \lambda \in  \R^B, Z \in T_{\bar{X}_{\theta}(T)}\mbox{Sp}(m)$ with
\begin{eqnarray*}
| \lambda |_{L^2} < \epsilon, \quad \left| Z-  G_{\theta}^{N,\bar{\theta}}(\lambda) \right| < c\, \epsilon^2,
\end{eqnarray*}
there are $\beta_1, \beta_2 \in \R^B$ such that  
\begin{eqnarray*}
Z = G_{\theta}^{N,\bar{\theta}} \bigl(\lambda+ \beta_1 + \beta_2\bigr),
\end{eqnarray*}
and
\begin{eqnarray*}
\beta_1 \in  \mbox{Ker} \left(D_{\lambda} G_{\theta}^{N,\bar{\theta}}\right), \quad \bigl| \beta_1\bigr| < \epsilon, \quad \bigl| \beta_2\bigr| < \epsilon^2.
\end{eqnarray*} 
Note that any 
$$
\sum_{b=1}^B \lambda_b \tilde{u}^b \quad \mbox{with} \quad  \lambda=(\lambda_1,...,\lambda_B) \in \R^B
$$
is a smooth control whose support is strictly contained in $[0,T]$. Then proceeding as in the proof of Proposition \ref{LIEPROP3} we conclude easily.

\section{Proof of Theorem \ref{THMmain}}\label{proofTHMmain}

We recall that given a geodesic $\gamma_{\theta}:[0,T] \rightarrow M$, an interval $[t_1,t_2] \subset [0,T]$ and $\rho>0$,  $\mathcal{C}_g\left(\gamma_{\theta}\bigl( [t_1,t_2]\bigr);\rho\right)$ stands for the open geodesic cylinder along $\gamma_{\theta}\bigl( [t_1,t_2]\bigr)$ of radius $\rho$, that is the open set defined by
\begin{multline*}
\mathcal{C}_g\left(\gamma_{\theta}\bigl( [t_1,t_2]\bigr);\rho\right) := \\
\Bigl\{ p \in M \, \vert \, \exists t\in (t_1,t_2) \mbox{ with } d_g\bigl(p,\gamma_{\theta}(t)\bigr)<\rho \mbox{ and }  d_g\bigl(p,\gamma_{\theta}([t_1,t_2])\bigr) = d_g\bigl(p,\gamma_{\theta}(t)\bigr) \Bigr\}.
\end{multline*}
The following holds: 

\begin{lemma} \label{tauT}
Let $(M,g)$ be a compact Riemannian manifold of dimension $\geq 2$. Then for every $T>0$, there exists $\tau_T \in (0,T)$ such that for every $\theta \in T_1M$, there are $\bar{t}\in [0,T-\tau_T]$ and $\bar{\rho}>0$ such that
$$
\mathcal{C}_g\Bigl( \gamma_{\theta}\left( \bigl[ \bar{t},\bar{t}+\tau_T\bigr]\right) ;\bar{\rho} \Bigr) \cap \gamma_{\theta}([0,T]) = \gamma_{\theta}\left( \bigl( \bar{t},\bar{t}+\tau_T\bigr)\right).
$$
\end{lemma}

\begin{proof}[Proof of Lemma \ref{tauT}]
Let $r_g>0$ be the injectivity radius of $(M,g)$, that is the supremum of $r>0$ such that any geodesic arc of length $r$ is minimizing between its end-points. We call self-intersection of the geodesic curve $\gamma_{\theta}([0,T])$ any $p\in M$ such that there are $t \neq t'$ in $[0,T]$ such that $\gamma_{\theta}(t)=\gamma_{\theta}(t') =p$. We claim that for every integer $k>0$ the number of self-intersection of a (non-periodic) geodesic of length $k \, r_g$ is bounded by 
$$
N(k) := \sum_{i=0}^{k-1} i = \frac{k(k-1)}{2}.
$$ 
We prove it by induction. Since any geodesic of length $r_g$ has no self-intersection, the result holds for $k=1$. Assume that we proved the result for $k$ and prove it for $(k+1)$. Let $\gamma :[0,(k+1)r_g]\rightarrow M$ be a unit speed geodesic of length $(k+1)r_g$.  The geodesic segment $\gamma([kr_g,(k+1)r_g])$ has no self-intersection but it could intersect the segment $\gamma ([0,kr_g])$. If the number of intersection of $\gamma([kr_g,(k+1)r_g])$ with $\gamma ([0,kr_g])$ is greater or equal than $(k+1)$, then there are $t_1 \neq t_2 \in [kr_g,(k+1)r_g]$, $i\in \{0,\ldots, k-1\}$, and $s_1, s_2 \in [ir_g, (i+1)r_g]$ such that 
$$
\gamma(t_1) = \gamma (s_1) \quad \mbox{and} \quad  \gamma(t_2) = \gamma (s_2).
$$
Since $\gamma$ is not periodic, this means that two geodesic arcs of length $\leq r_g$ join $\gamma(t_1)$ to $\gamma (t_2)$, a contradiction. We infer that the number of self-intersection of $\gamma$ is bounded by $N(k)+k$, and in turn that it is bounded by $N(k+1)$. We deduce that for every integer $k\geq 2$, all the disjoint open intervals
$$
I_i := \left( i \frac{kr_g}{N(k)}, (i+1) \frac{kr_g}{N(k)}\right) \qquad i=0, \ldots, N(k)-1
$$
can not contain a point of self-intersection of a unit speed geodesic $\gamma :[0,kr_g]\rightarrow M$. Hence for every  unit speed geodesic $\gamma :[0,kr_g]\rightarrow M$ there is $i \in \{0, \ldots, N(k)-1\}$ such that no self-intersection of $\gamma$ is contained in the closed interval 
$$
\left[ i \frac{kr_g}{N(k)}, (i+1) \frac{kr_g}{N(k)}\right].
$$
We conclude easily. 
\end{proof}

Let $T>0$ be fixed,  $\tau_T \in (0,T) $ given by Lemma \ref{tauT} and  $\gamma_{\theta}:[0,T] \rightarrow M$ be a unit speed geodesic of length $T$. Then there are $\bar{t}\in [0,T-\tau_T]$ and $\rho>0$ such that
$$
\mathcal{C}_g\Bigl( \gamma_{\theta}\left( \bigl[ \bar{t},\bar{t}+\tau_T\bigr]\right) ;\rho \Bigr) \cap \gamma_{\theta}([0,T]) = \gamma_{\theta}\left( \bigl( \bar{t},\bar{t}+\tau_T\bigr)\right).
$$
Set 
$$
 \bar{\theta} = \left( \bar{p},\bar{v}\right) :=  \left( \gamma_{\theta}(\bar{t}),\dot{ \gamma}_{\theta}(\bar{t})\right)  \quad \tilde{\theta} = \left(\tilde{p},\tilde{v}\right) := \left( \gamma_{\theta}(\bar{t}),\dot{\gamma}_{\theta}(\bar{t}+\tau_T)\right),
 $$
 $$
  \quad \theta_T = (p_T,v_T) := \left(\gamma_{\theta}(T), \dot{\gamma}_{\theta}(T) \right),
$$
and consider local transverse sections $\Sigma_0, \bar{\Sigma}, \tilde{\Sigma}, \Sigma_T \subset T_1M$ respectively tangent to  $N_{\theta}, N_{\bar{\theta}}, N_{\tilde{\theta}}, N_{\theta_T}$. Then we have 
$$
P_g(\gamma)(T)=D_{\theta}\P_g(\Sigma_0,\Sigma_T,\gamma) = D_{\tilde{\theta}}\P_g\bigl(\tilde{\Sigma},\Sigma_T,\gamma\bigr) \circ D_{\bar{\theta}}\P_g\bigl(\bar{\Sigma},\tilde{\Sigma},\gamma\bigr)\circ D_{\theta}\P_g \bigl(\Sigma_0,\bar{\Sigma},\gamma\bigr).
$$
Since the sets of symplectic endomorphism of $\mbox{Sp}(n-1)$ of the form $D_{\tilde{\theta}}\P_g\bigl(\tilde{\Sigma},\Sigma_T,\gamma\bigr)$ and $D_{\theta}\P_g \bigl(\Sigma_0,\bar{\Sigma},\gamma\bigr)$ (that is the differential of Poincar\'e maps associated with geodesics of lengths $T-\bar{t}-\tau_T$ and $\bar{t}$) are compact and the left and right translations in $\mbox{Sp}(n-1)$ are diffeomorphisms,  it is sufficient to prove Theorem \ref{THMmain} with $\tau_T=T$. More exactly, it is sufficient to show that there are $\delta_{T}, K_T>0$ such that for every $\delta \in (0, \delta_{T})$ and every $\rho>0$,  the following property holds:\\
Let $\gamma_{\theta}:[0,\tau_T] \rightarrow M$ be a geodesic in $M$, $U$ be the open ball centered at $P_g(\gamma)(\tau_T)$ of radius $\delta$ in $\mbox{Sp}(n-1)$.  Then for each symplectic map $A \subset U$ there exists a $C^{\infty}$ metric $h$ in $M$ that is conformal to $g$, $h_{p}(v,w) = (1+\sigma(p))g_{p}(v,w)$, such that 
\begin{enumerate}
\item The geodesic $\gamma_{\theta} : [0,\tau_T] \longrightarrow M$ is still a geodesic of $(M,h)$, 
\item $\mbox{Supp} (\sigma) \subset \mathcal{C}_g\left( \gamma_{\theta}\left( [ 0,\tau_T]\right) ;\rho \right)$,
\item $P_{h}(\gamma_{\theta})(\tau_T) = A$,
\item the $C^{2}$ norm of the function $\sigma$ is less than $K_T \sqrt{\delta}$.
\end{enumerate}
Set $\tau:=\tau_T$ and let $\gamma:[0,\tau] \rightarrow M$ a geodesic in $M$ be fixed, we consider  a Fermi coordinate system $\Phi(t,x_{1},x_{2},..,x_{n-1})$, $t \in (0,\tau)$, $(x_{1},x_{2},..,x_{n-1}) \in (-\delta, \delta)^{n-1}$ along $\gamma([0,\tau])$, where $t$ is the arc length of $\gamma$, and the coordinate vector fields $e_{1}(t), \ldots, e_{n-1}(t)$ of the system are orthonormal and parallel along $\gamma$. Let us consider  the family of smooth functions $\{P_{ij}\}_{i,j=1, \ldots, n-1} : \R^{n-1} \rightarrow \R$ defined by
$$
P_{ij} \bigl(y_{1},y_{2},..,y_{n-1}\bigr) := y_{i}y_{j} \, Q \bigl(|y|\bigr) \qquad \forall i\neq j \in \{1,\ldots,n-1\}
$$
 and
$$ 
P_{ii}(y_{1},y_{2},..,y_{n-1}) :=   \frac{y_{i}^{2}}{2} \, Q \bigl(|y|\bigr)  \qquad \forall i \in \{1,\ldots,n-1\},
$$
for every $y=\left(y_{1},y_{2},..,y_{n-1}\right) \in \R^{n-1}$ where $Q:[0,+\infty) \rightarrow [0,+\infty)$ is a smooth cutoff function satisfying
$$
\left\{
\begin{array}{rcl}
Q(\lambda) = 1 & \mbox{ if } & \lambda \leq 1/3 \\
Q(\lambda) = 0 & \mbox{ if } & \lambda \geq 2/3.
\end{array}
\right.
$$
Given a radius $\rho>0$ with $\mathcal{C}_g\left( \gamma \left( [ 0,\tau]\right) ;\rho \right) \subset \Phi(\left(0, \tau) \times (-\delta, \delta)^{n-1}\right) $ and a family of smooth function $u=(u_{ij})_{i\leq j=1,\ldots,n-1}: [0,\tau] \rightarrow \R$ such that
$$
\mbox{Supp} \bigl( u_{ij}\bigr) \subset (0,\tau) \qquad \forall i\leq j\in \{1,\ldots,n-1\},
$$
we define a family of smooth perturbations 
$$
\left\{\sigma_{ij}^{\rho,u}\right\}_{i\leq j = 1,\ldots,n-1} \, : \, M \, \longrightarrow \, \R
$$
with support in $\Phi(\left(0, \tau) \times (-\delta, \delta)^{n-1}\right) $ by 
$$ 
\sigma_{ij}^{\rho,u} \left(\Phi \bigl(t,x_{1},x_{2},..,x_{n-1}\bigr) \right) := \rho^{2}\, \, u_{ij}(t) P_{ij} \left(\frac{x_{1}}{\rho}, \frac{x_{2}}{\rho}, .., \frac{x_{n-1}}{\rho}\right),
$$
for every $p =    \Phi \bigl(t,x_{1},x_{2},..,x_{n-1}\bigr)   \in \Phi \left((0, \tau) \times (-\delta, \delta)^{n-1}\right)$ and we define $\sigma^{\rho,u} : M \rightarrow \R$ by
$$
\sigma^{\rho,u}   := \sum_{i,j=1, i\leq j}^{n-1} \sigma_{ij}^{\rho,u}.
$$
The following result follows by construction, its proof is left to the reader. The notation $\partial_l$ with $l=0,1, \ldots, n-1$ stands for the partial derivative in coordinates $x_0=t, x_1, \ldots, x_{n-1}$ and $H\sigma^{\rho,u}$ denotes the Hessian of $\sigma^{\rho,u}$ with respect to $g$.

\begin{lemma} \label{LEMlocalpert}
The following properties hold:
\begin{enumerate}
\item $\mbox{Supp} (\sigma^{\rho,u}) \subset \mathcal{C}_g\left( \gamma \left( [ 0,\tau]\right) ;\rho \right) $, 
\item $\sigma^{\rho,u}(\gamma(t)) = 0$ for every $t \in (0, \tau)$,
\item $\partial_l \sigma^{\rho,u}(\gamma(t)) = 0$ for every $t \in (0, \tau)$ and $l =0, 1, \ldots, n-1$,
\item $ \left( H\sigma^{\rho,u}\right)_{i,0} (\gamma(t))  =  0$ for every $t \in (0, \tau)$ and $i=1, \ldots, n-1$,
\item $ \left( H\sigma^{\rho,u}\right)_{i,j} (\gamma(t)) =  u_{ij}(t)$ for every $t \in (0, \tau)$ and $i, j =1, \ldots, n-1$,
\item $ \| \sigma^{\rho,u} \|_{C^{2}} \leq C \|u\|_{C^2}$ for some universal constant $C>0$.
\end{enumerate}
\end{lemma}

Set $m=n-1$ and $k:=m(m+1)/2$. Let $u=(u_{ij})_{i\leq j=1,\ldots,n-1}: [0,\tau] \rightarrow \R$  be a family of smooth functions with support strictly contained in $(0,\tau)$ and $\rho \in (0,\bar{\rho})$ be fixed, using the previous notations we set the conformal metric
$$
h := \left( 1 + \sigma^{\rho,u} \right)^2 \, g.
$$ 
We denote by $\langle \cdot, \cdot \rangle, \nabla, \Gamma, \H, \Rm$ respectively the scalar product, gradient, Christoffel symbols,
Hessian and curvature tensor associated with $g$. With the usual notational conventions
of Riemannian geometry (as in \cite{docarmo}), in components we have
$$
\left\{
\begin{array}{rcl}
\G_{ij}^k & = & \frac12 \Bigl( \partial_i g_{jm} + \partial_j g_{im} - \partial_m g_{ij}\Bigr) g^{mk}\\
(\H f)_{ij} & = & \partial_{ij} f - \G_{ij}^k \partial_k f,
\end{array}
\right.
$$
where $(g^{k\ell})$ stands for the inverse of $(g_{k\ell})$, and we use Einstein's convention of
summation over repeated indices. We shall use a superscript $h$ to denote the same objects when they are associated with the
metric $h$.  As usual $\delta_{ij}=\delta^{ij}=\delta_i^j$ will be 1 if
$i=j$, and 0 otherwise. The Christoffel symbols are modified as follows by a conformal
change of metrics: if $h = e^{2f}g$ then (see for example \cite{kn:Kulkarni})
$$
(\G^h)_{ij}^k = \G_{ij}^k + \bigl( \partial_i f \delta_j^k + \partial_j f \delta_i^k - \partial_m f g_{ij} g^{mk}\bigr).
$$
Thus, since $f= \ln (1+\sigma^{\rho,u})$ and its derivatives $\partial_0 f, \partial_1 f, \ldots, \partial_{n-1} f$ vanish along $\gamma ([0,\tau])$ (by Lemma \ref{LEMlocalpert} (2)-(3)), the Christoffel symbols of $h$ and $g$ coincide along $\gamma$. Then the family of tangent vectors $e_0(t)=\dot{\gamma}(t), e_{1}(t), \ldots, e_{n-1}(t)$ is still a family which is orthonormal and parallel along $\gamma  ([0,\tau])$. Moreover, if $h = e^{2f}g$ then the curvature tensor $\Rm^h, \Rm$ respectively of $h$ and $g$ satisfy 
$$
e^{-2f}\, \left\langle \Rm^h (u,v)\, v,w\right\rangle^h \\
=  \left\langle \Rm (u,v)\, v,w\right\rangle - \H f (u,w),
$$
at any $p\in M$ where $\nabla f$ vanishes and any tangent vectors $u, v, w \in T_pM$ such that  $u,w \perp v$ and $Hf(v,\cdot)=0$. By Lemma \ref{LEMlocalpert} (2)-(5), we infer that along $\gamma ([0,\tau])$, we have for every $i,j \in \{1, \ldots, n-1\}$ and every $t\in [0,\tau]$,
\begin{eqnarray}\label{Rij}
R^h_{ij}(t) & := & \left\langle \Rm^h_{\gamma(t)} \left(e_i(t), \dot{\gamma}(t)\right) \dot{\gamma}(t),e_j(t) \right\rangle_{\gamma(t)}^h\\
& = & \left\langle \Rm_{\gamma(t)} \left( e_i(t), \dot{\gamma}(t)\right) \dot{\gamma}(t),e_j(t) \right\rangle_{\gamma(t)} - u_{ij}(t) \nonumber \\
& = & R_{ij}(t) - u_{ij}(t)\nonumber,
\end{eqnarray}
with 
\begin{eqnarray}\label{5mars1}
R_{ij}(t)  :=  \left\langle \Rm_{\gamma(t)} \left(e_i(t), \dot{\gamma}(t)\right) \dot{\gamma}(t),e_j(t) \right\rangle_{\gamma(t)}.
\end{eqnarray}
By the above discussion,  $\gamma$ is still a geodesic with respect to $h$ and by construction (Lemma \ref{LEMlocalpert} (1)) the support of $\sigma^{\rho,u}$ is contained in a cylinder of radius $\rho$, so properties (1) and (2) above are satisfied. it remains to study the effect of $\sigma^{\rho,u}$ on the symplectic mapping $P_h(\gamma)(\tau)$. By the Jacobi equation, we have
$$
P_h(\gamma)(\tau)(J(0),\dot{J}(0))=(J(\tau),\dot{J}(\tau)),
$$
where $J: [0,\tau] \rightarrow \R^m$ is solution to the Jacobi equation 
$$
\ddot{J}(t) + R^h(t) J(t) =0 \qquad \forall t \in [0,\tau],
$$
where $R^h(t)$ is the $m\times m$ symmetric matrix whose coefficients are given by (\ref{Rij}). In other terms, $P_h(\gamma)(\tau)$ is equal to the $2m\times 2m$ symplectic matrix $X(\tau)$ given by the solution $X:[0,\tau] \rightarrow \mbox{Sp}(m)$ at time $\tau$ of the following Cauchy problem (compare \cite[Sect. 3.2]{rr12} and \cite{lazrag14}):
\begin{eqnarray}\label{syscontrol2}
\dot{X}(t) = A(t) X(t) + \sum_{i\leq j=1}^m u_{ij}(t) \mathcal{E}(ij) X(t) \quad \forall t\in [0,\tau], \quad X(0)=I_{2m},
\end{eqnarray}
where the $2m \times 2m$ matrices $A(t), \mathcal{E}(ij) $ are defined by ($R(t)$ is the $m\times m$ symmetric matrix whose coefficients are given by (\ref{5mars1}))
$$
A(t) :=  \left( \begin{matrix} 0 & I_m \\ -R(t) & 0 \end{matrix} \right) \qquad \forall t \in [0,\tau]
$$
and
$$
\mathcal{E}(ij) :=   \left( \begin{matrix} 0 & 0 \\ E(ij) & 0 \end{matrix} \right),
$$
where the $E(ij), 1 \leq i\leq j\leq m$ are the symmetric $m \times m$ matrices defined by
$$
\mbox{and} \quad \left(E(ij)\right)_{k,l} =   \delta_{ik} \delta_{jl} + \delta_{il} \delta_{jk}  \qquad \forall i, j =1, \ldots,m.
$$
Since our control system has the form (\ref{syscontrol}), all the results gathered in Section \ref{prelcontrol} apply. So, Theorem \ref{THMmain} will follow from Proposition \ref{LIEPROP3bis}. First by compactness of $M$ and regularity of the geodesic flow, the compactness assumptions in Proposition \ref{LIEPROP3bis} are satisfied. It remains to check that assumptions (\ref{ASSBBbis}), (\ref{conditionLIEPROP4bis}) and (\ref{conditionLIEPROP3bis}) hold.

First we check immediately that 
$$
\mathcal{E}(ij) \mathcal{E}(kl) =0 \qquad \forall  i,j,k, l \in \{1, \ldots,m\} \mbox{ with } i\leq j, \, k\leq l.
$$
So, assumption (\ref{ASSBBbis}) is satisfied. Since the $\mathcal{E}(ij)$ do not depend on time, we check easily that the matrices $B_{ij}^0, B_{ij}^1, B_{ij}^2$ associated to our system are given by (remember that we use the notation $[B,B'] := BB'-B'B$)
$$
\left\{
\begin{array}{l}
B_{ij}^0 (t) = B_{ij} := \mathcal{E}(ij) \\
B_{ij}^1(t) = \left[ \mathcal{E}(ij), A(t)\right] \\
B_{ij}^2(t) =   \left[ \left[ \mathcal{E}(ij), A(t)\right], A(t) \right],
\end{array}
\right.
$$
for every $t \in [0,\tau]$ and any $i, j =1, \ldots,m$ with $i\leq j$. An easy computation yields for any $i, j =1, \ldots,m$ with $i\leq j$ and any $t\in [0,\tau]$,
$$
B_{ij}^1(t) = \left[\mathcal{E}(ij), A (t)\right] = \left( \begin{matrix} -E(ij) & 0 \\ 0 & E(ij) \end{matrix} \right)
$$
and
$$
B_{ij}^2(t) =  \left[ \left[ \mathcal{E}(ij), A(t) \right], A(t) \right] =  \left( \begin{matrix} 0 & -2E(ij) \\ -E(ij)R(t)-R(t) E(ij) & 0 \end{matrix} \right).
$$
Then we get for any $i, j =1, \ldots,m$ with $i\leq j$,
$$
\left[B_{ij}^1(0),B_{ij}\right]=2 \left( \begin{matrix} 0 & 0 \\ (E(ij))^2 & 0 \end{matrix} \right) \in \mbox{Span} \Bigr\{B_{rs}^0(0) \, \vert \,r\leq s \Bigr\}
$$
and
$$
\left[B_{ij}^2(0),B_{ij}\right]=2 \left( \begin{matrix} -(E(ij))^2 & 0 \\ 0 & (E(ij))^2 \end{matrix} \right) \in \mbox{Span} \Bigr\{B_{rs}^1(0) \, \vert \, r \leq s \Bigr\}.
$$
So assumption (\ref{conditionLIEPROP4bis}) is satisfied. It remains to show that (\ref{conditionLIEPROP3bis}) holds. We first notice that  for any $i, j, k,l =1, \ldots,m$ with $i\leq j, k\leq l$, we have
\begin{eqnarray*}
 \left[ B_{ij}^1(0), B_{kl}^1(0)\right] & = & \Bigl[ \left[ \mathcal{E}(ij), A(0)\right], \left[ \mathcal{E}(kl), A(0)\right] \Bigr] \\
& = & \left( \begin{matrix} \left[E(ij),E(kl)\right] & 0 \\ 0 &  \left[E(ij),E(kl)\right] \end{matrix} \right),
\end{eqnarray*}
with 
\begin{equation}\label{2702}
 \left[E(ij),E(kl)\right]  = \delta_{il} F(jk) + \delta_{jk} F(il) + \delta_{ik} F(jl) + \delta_{jl} F(ik),
\end{equation}
where $F(pq)$ is the $m\times m$ skew-symmetric matrix defined by
$$
\left( F(pq)\right)_{rs} := \delta_{rp}\delta_{sq} - \delta_{rq}\delta_{sp}.
$$
It is sufficient to show that the space $S \subset M_{2m}(\R)$ given by
$$
S:=\mbox{Span}  \Bigl\{B_{ij}^0(0),B_{ij}^1(0),B_{ij}^2(0),[B_{kl}^1(0),B_{rr'}^1(0)] \, \vert \, i,j,k,l,r,r' \Bigr\} \subset  T_{I_{2m}}\mbox{Sp}(m)
$$ 
has dimension $p:=2m(2m+1)/2$. First since the set matrices $ \mathcal{E}(ij)$ with $i, j=1, \ldots,m$ and $i\leq j$ forms a basis of the vector space of $m\times m$ symmetric matrices $\mathcal{S}(m)$ we check easily by the above formulas that the vector space 
$$
S_1 :=\mbox{Span}    \Bigl\{   B_{ij}^0, B_{ij}^2 (0)  \, \vert \, i, j\Bigr\}  =       \mbox{Span} \Bigl\{ \mathcal{E}(ij), \left[ \left[ \mathcal{E}(ij), A(t) \right], A(t) \right] \, \vert \, i, j\Bigr\}
$$
has dimension $2 (m(m+1)/2)=m(m+1)$. We check easily that the vector spaces 
$$
S_2 := \mbox{Span}    \Bigl\{  B_{ij}^1 (0)  \, \vert \, i, j\Bigr\}  = \mbox{Span}   \Bigl\{  \left[ \mathcal{E}(ij), A(0)\right] \, \vert \, i,j \Bigr\}
$$
and
\begin{multline*}
S_3 :=    \mbox{Span}    \Bigl\{      \left[ B_{ij}^1(0), B_{kl}^1(0)\right]      \, \vert \, i, j,k,l \Bigr\} =  \\   \mbox{Span}    \Bigl\{  \Bigl[ \left[ \mathcal{E}(ij), A(0)\right], \left[ \mathcal{E}(kl), A(0)\right] \Bigr]   \, \vert \, i,j, k, l\Bigr\}
\end{multline*}
are orthogonal to $S_1$ with respect to the scalar product $P \cdot Q=\mbox{tr} (P^*Q)$. So, we need to show that $S_2 + S_3$ has dimension  $p-m(m+1)=m^2$. By the above formulas, we have 
$$
S_2 := \mbox{Span}    \left\{  \left( \begin{matrix} -E(ij) & 0 \\ 0 & E(ij) \end{matrix} \right)  \, \vert \, i, j\right\} 
$$
and
$$
S_3 :=    \mbox{Span}    \left\{    \left( \begin{matrix} \left[E(ij),E(kl)\right] & 0 \\ 0 &  \left[E(ij),E(kl)\right] \end{matrix} \right)  \, \vert \, i,j, k, l\right\},
$$
and in addition $S_2$ and $S_3$ are orthogonal. The first space $S_2$ has the same dimension as $\mathcal{S}(m)$, that is $m(m+1)/2$. Moreover, by  (\ref{2702}) for every $i\neq j$, $k=i,$ and $l\notin \{i,j\}$, we have 
$$
 \left[E(ij),E(kl)\right]  = F(jl).
$$
The space spanned  by the matrices of the form 
$$
 \left( \begin{matrix} F(jl) & 0 \\ 0 &  F(jl) \end{matrix} \right),
$$
with $1 \leq j < l \leq m$ has dimension $m(m-1)/2$. This shows that $S_3$ has dimension at least $m(m-1)/2$ and so $S_2\oplus S_3$ has dimension $m^2$. This concludes the proof of Theorem \ref{THMmain}.

\section{Proofs of Theorems \ref{THM1} and \ref{THM2}}\label{proofTHM1THM2}

Let us start with the proof of Theorem \ref{THM1}, namely, if the periodic orbits of the geodesic flow of a smooth compact manifold $(M,g)$ of dimension $\geq 2$ are $C^{2}$-persistently hyperbolic from Ma\~{n}\'{e}'s viewpoint then the closure of the set of periodic orbits is a hyperbolic set. Recall that an invariant set $\Lambda$ of a smooth flow $\psi_{t}: Q \longrightarrow Q$ acting
without singularities on a complete manifold $Q$ is called hyperbolic if there exist constants, $C>0$, $\lambda \in (0,1)$, and
a direct sum decomposition $T_{p}Q= E^{s}(p)\oplus E^{u}(p) \oplus
X(p)$ for every $p \in \Lambda $, where $X(p)$ is the subspace tangent to
the orbits of $\psi_{t}$, such that
\begin{enumerate}
\item $\parallel D\psi_{t}(W)\parallel \leq C\lambda^{t} \parallel W\parallel$ for every $W \in E^{s}(p)$ and $ t \geq 0$,
\item $\parallel D\psi_{t}(W)\parallel \leq C\lambda^{-t} \parallel$ for every $W \in E^{u}(p)$ and $ t \leq 0$.
\end{enumerate}
In particular, when the set $\Lambda$ is the whole $Q$ the flow is called Anosov. The proof follows the same steps of the proof of Theorem B in \cite{ruggiero91} where
the same conclusion is obtained supposing that the geodesic flow is $C^{1}$ persistently expansive in the family of Hamiltonian flows.

\subsection{Dominated splittings and hyperbolicity}

Let $F^{2}(M,g)$ be the set of Riemannian metrics in $M$ conformal to $(M,g)$ endowed with the $C^{2}$ topology such that all closed orbits of their
geodesic flows are hyperbolic.

The first step of the proof of Theorem \ref{THM1} is closely related with the notion of dominated splitting introduced by Ma\~{n}\'{e}.

\begin{definition}
Let $\phi_{t}: Q \longrightarrow  Q$ be a smooth non-singular flow acting on a
complete Riemannian manifold $Q$ and let $\Omega \subset Q$ be and invariant set. We say that
$\Omega$ has a dominated splitting in $\Omega$ if there exist constants $\delta \in (0,1)$, $m>0$, and invariant subspaces
$S(\theta), U(\theta)$ in $T_{\theta}\Omega$ such that for every $\theta \in \Omega$,
\begin{enumerate}
\item If $X(\theta)$ is the unit vector tangent to the flow then $S(\theta)\oplus U(\theta) \oplus X(\theta) = T_{\theta}Q$,
\item $\parallel D_{\theta}\phi_{m}\vert_{S(\theta)} \parallel \cdot \parallel D_{\phi_{m}(\theta)}\phi_{-m}\vert_{U(\phi_{m}(\theta)} \parallel \leq \delta.$
\end{enumerate}
\end{definition}

The invariant splitting of an Anosov flow is always dominated, but the converse may not be true in general. However, for geodesic
flows the following statement holds

\begin{theorem} \label{domination}
Any continuous, Lagrangian, invariant, dominated splitting in a compact invariant set for the geodesic flow of a smooth compact Riemannian manifold is a hyperbolic splitting.
Therefore, the existence of a continuous Lagrangian invariant dominated splitting in the whole unit tangent bundle is equivalent to the Anosov property in the family of
geodesic flows.
\end{theorem}

This statement is proved in \cite{ruggiero91} not only for geodesic flows but for symplectic diffeomorphisms. Actually, the statement
extends easily to a Hamiltonian flow in a nonsingular energy level (see also Contreras \cite{kn:Contreras}).

The following step of the proof of Theorem \ref{THM1} relies on the connection between persistent hyperbolicity of periodic orbits and 
the existence of invariant dominated splittings. One of the most remarkable facts about Ma\~{n}\'{e}'s work 
about the stability conjecture (see Proposition II.1 in \cite{kn:Manest}) is to show that persistent hyperbolicity of families of linear maps is connected to dominated splittings, 
the proof is pure generic linear algebra (see Lemma II.3 in \cite{kn:Manest}). Then Ma\~{n}\'{e} observes that Franks' Lemma allows to reduce the study of persistently 
hyperbolic families of periodic orbits of diffeomorphisms to persistently hyperbolic families of linear maps. Let us explain briefly  Ma\~{n}\'{e}'s 
result and see how its combination with Franks' Lemma for geodesic flows implies Theorem \ref{THM1}. 

Let $GL(n)$ be the group of linear isomorphisms of $\mathbb{R}^{n}$. Let $\psi : \mathbb{Z} \longrightarrow GL(n)$ be a sequence of such isomorphisms,  we denote by $E^{s}_{j}(\psi) $ the set of vectors $v \in \mathbb{R}^{n}$ such that 
$$
\sup_{n \geq 0} \Bigl\{ \bigl\| \left(\Pi_{i=0}^{n} \psi_{j+i}\right) v \bigr\|  \Bigr\} <\infty,
$$ 
and by $E^{u}_{j}(\psi) $ the set of vectors $v \in \mathbb{R}^{n}$ such that 
$$
\sup_{n \geq 0} \Bigl\{ \bigl\| \left(\Pi_{i=0}^{n} \psi_{j-1-i} \right)^{-1} v \bigr\|  \Bigr\} <\infty. 
$$ 
Let us say that the sequence $\psi$ is hyperbolic if $E^{s}_{j}(\psi)  \bigoplus E^{u}_{j}(\psi) = \mathbb{R}^{n}$ for every $j \in \mathbb{Z}$. 
Actually, this definition is equivalent to require the above direct sum decomposition for some $j$. A periodic sequence $\psi$ is characterized by 
the existence of $n_{0}>0$ such that $\psi_{j+n_{0}} = \psi_{j}$ for every $j$. It is easy to check that the hyperbolicity of a periodic sequence $\psi$ 
is equivalent to the classical hyperbolicity of the linear map $\prod_{j=0}^{n_{0}-1} \psi_{j}$. Now, let 
$$
\Bigl\{\psi^{\alpha}, \alpha \mbox{ } \in \mbox{ }\Lambda \Bigr\}
$$ 
be a family of periodic sequences of linear maps indexed in a set $\Lambda$. Let us define the distance $d(\psi, \eta)$ between two families of 
periodic sequences indexed in $\Lambda$ by 
$$ 
d(\psi, \eta) = \sup_{n \in \mathbb{Z}, \alpha \in \Lambda} \Bigl\{ \parallel \psi_{n}^{\alpha} - \eta_{n}^{\alpha} \parallel \Bigr\} .
$$ 
We say that the family $\{\psi^{\alpha}, \alpha \mbox{ } \in \mbox{ }\Lambda \}$ is hyperbolic if every sequence in the family is hyperbolic. 
Let us call by periodically equivalent two families $\psi^{\alpha}$, $\eta^{\alpha}$ for which given any $\alpha$, the minimum periods of $\psi^{\alpha}$ 
and $\eta^{\alpha}$ coincide. Following Ma\~{n}\'{e}, we say that the family $\{\psi^{\alpha}, \alpha \mbox{ } \in \mbox{ }\Lambda \}$ is uniformly hyperbolic 
if there exists $\epsilon >0$ such that every periodically equivalent family $\eta^{\alpha}$ such that $d(\psi, \eta) < \epsilon$ is also hyperbolic. 
The main result concerning uniformly hyperbolic families of linear maps is the following symplectic version of Lemma II.3 in \cite{kn:Manest}. 

\begin{theorem} \label{persistentlinearhyp}
Let $\{\psi^{\alpha}, \alpha \mbox{ } \in \mbox{ }\Lambda \}$ be a uniformly hyperbolic family of periodic linear sequences of 
symplectic isomorphisms of $\mathbb{R}^{n}$. Then there exist constants $K>0$, $m \in \mathbb{N}$, and $\lambda \in (0,1)$ such that : 
\begin{enumerate}
\item If $\alpha \in \Lambda$ and $\psi^{\alpha}$ has minimum period $n \geq m$, then 
$$ \prod_{j=0}^{k-1} \bigl\| (\Pi_{i=0}^{m-1} \psi_{mj+i}^{\alpha})\vert_{E^{s}_{mj}(\psi^{\alpha})} \bigr\| \leq K\lambda^{k} , $$ 
and 
$$ \prod_{j=0}^{k-1} \bigl\| (\Pi_{i=0}^{m-1} \psi_{mj+i}^{\alpha})^{-1}\vert_{E^{u}_{mj}(\psi^{\alpha})} \bigr\| \leq K\lambda^{k} , $$ 
where $ k $ is the integer part of $\frac{n}{m}$. 
\item For all $\alpha \in \Lambda$, $j \in \mathbb{Z}$, 
$$ 
\bigl\| (\Pi_{i=0}^{m-1} \psi_{j+i}^{\alpha})\vert_{E^{s}_{j}(\psi^{\alpha})} \bigr\| \cdot \bigl\| (\Pi_{i=0}^{m-1} \psi_{j+i}^{\alpha})^{-1}\vert_{E^{u}_{j}(\psi^{\alpha})} \bigr\| \leq \lambda. $$
\item For every $\alpha \in \Lambda$ 
$$ \limsup_{n \rightarrow +\infty} \frac{1}{n}\sum_{j=0}^{n-1} \ln\left( \bigl\| (\Pi_{i=0}^{m-1} \psi_{mj+i}^{\alpha})\vert_{E^{s}_{mj}(\psi^{\alpha}) }\bigr\| \right) < 0$$ 
and 
$$ \limsup_{n \rightarrow +\infty} \frac{1}{n}\sum_{j=0}^{n-1} \ln \left(\bigl\| (\Pi_{i=0}^{m-1} \psi_{mj+i}^{\alpha})^{-1}\vert_{E^{u}_{m(j+1)}(\psi^{\alpha})} \bigr\| \right) < 0. $$ 
\end{enumerate}
\end{theorem}

At the end of the section we shall give an outline of the proof of Theorem \ref{persistentlinearhyp} based on Ma\~{n}\'{e}'s Lemma II.3 in \cite{kn:Manest} 
which is proved for linear isomorphisms without the symplectic assumption. 

Now, we are ready to combine Franks' Lemma from Ma\~{n}\'{e}'s viewpoint and Theorem \ref{persistentlinearhyp} to get a geodesic flow version of 
Theorem \ref{persistentlinearhyp}.  

\begin{lemma} \label{minimumperiod}
Let $(M,g)$ be a compact Riemannian manifold. Then there exists $T_{g}>0$ such that every closed geodesic has period greater than $T_{g}$.
\end{lemma}

The proof is more or less obvious from the flowbox lemma since the geodesic flow has no singularitites and the unit tangent bundle of $(M,g)$ is compact.\\

Let $\mbox{Per}(g)$ be the set of periodic points of the geodesic flow of $(M,g)$. Given a periodic point $\theta \in \mbox{Per}(g)$ with period $T(\theta)$, consider a family of local sections $\Sigma^{\theta}_{i}$, $i = 0, 1,.., k_{\theta}=[\frac{T(\theta)}{T_{g}}]$,  where $[\frac{T(\theta)}{T_{g}}]$ is the integer part of $\frac{T(\theta)}{T_{g}}$, with the following properties: 
\begin{enumerate}
\item $\Sigma^{\theta}_{i}$ contains the point $\phi_{iT_{g}}(\theta)$ for every $i = 0,1,.., k_{\theta}-1$, 
\item $\Sigma^{\theta}_{i}$ is perpendicular to the geodesic flow at $\phi_{iT_{g}}(\theta)$ for every $ i$.
\end{enumerate}
Let us consider the sequence of symplectic isomorphisms 
$$
\psi_{\theta, g} = \Bigl\{ A_{\theta, i , g} , \, i \in \mathbb{Z} \Bigr\} 
$$ 

\begin{enumerate}
\item For $i = nk_{\theta} + s$, where $n \in \mathbb{Z}$, $0\leq s <k_{\theta}-1$, let 
$$A_{\theta, i , g} = D_{\phi_{sT_{g}}(\theta)}\phi_{T_{g}} : T_{\phi_{sT_{g}}(\theta)} \Sigma^{\theta}_{s} \longrightarrow T_{\phi_{(s+1)T_{g}}(\theta)}\Sigma^{\theta}_{s+1},$$ 

\item For $ i = nk_{\theta} -1$, where $n \in \mathbb{Z}$, let 
$$A_{\theta,i , g}=  D_{\phi_{(k_{\theta}-1)T_{g}}(\theta)}\phi_{T_{g}+ r_{\theta}}: 
 T_{\phi_{(k_{\theta}-1)T_{g}}(\theta)} \Sigma^{\theta}_{(k_{\theta}-1)} \longrightarrow T_{\theta}\Sigma^{\theta}_{0}$$ 
where $T(\theta) = k_{\theta} T_{g} + r_{\theta}$. 
\end{enumerate}
Notice that the sequence  $\psi_{\theta, g}$ is periodic and let 
$$ 
\psi_{g} = \Bigl\{ \psi_{\theta,g}, \theta \in \mbox{Per}(g) \Bigr\}.
$$ 
The family $\psi_{g}$ is a collection of periodic sequences, and by Franks' Lemma from Ma\~{n}\'{e}'s viewpoint (Theorem \ref{THMmain}) we have 

\begin{lemma} \label{Franks1}
Let $(M,g)$ be a compact Riemannian manifold. If  $(M,g)$ is in the interior of $F^{2}(M,g)$ then the family $\psi_{g}$ is uniformly hyperbolic. 
\end{lemma}

\begin{proof}
Let $\delta_{T_{g}} >0$, $K_{T_{g}}$, be given in Franks' Lemma, Theorem \ref{THMmain}.

If $(M,g)$ is in the interior of $F^{2}(M,g)$ then there exists an open $C^{2}$ neighborhood $U$ of $(M,g)$ in the set of metrics which are 
conformally equivalent to $(M,g)$ such that every closed orbit of the geodesic flow of $(M,h) \in U$ is hyperbolic.  In particular, given a periodic 
point $\theta \in T_{1}M$ for the geodesic flow of $(M,g)$, the set of metrics $(M,h_{\theta}) \in U$ for which the orbit of $\theta$ is still a periodic 
orbit for the geodesic flow of $(M,h_{\theta})$ have the property that this orbit is hyperbolic as well for the $h_{\theta}$-geodesic flow. By Theorem 
\ref{THMmain}, for any $\delta \in (0,\delta_{T_{g}})$, the $(K_{T_{g}}\sqrt{\delta})$-$C^{2}$ open neighborhood of the metric $(M,g)$ in 
the set of its conformally equivalent metrics covers a $\delta$-open neighborhood of symplectic linear transformations of the derivatives of the Poincar\'{e} maps between 
the sections $\Sigma^{\theta}_{s}$, $\Sigma^{\theta}_{s+1}$ defined above. Then consider $\delta >0$ such that the 
$(K_{T_{g}}\sqrt{\delta})$-$C^{2}$ open neighborhood of the metric $(M,g)$ is contained in $U$, and we get that the family 
$A_{\theta, i , g}$ is uniformly hyperbolic. Since this holds for every periodic point $\theta$ for the geodesic flow of $(M,g)$ 
the family $\psi_{g}$ is uniformly hyperbolic. 
\end{proof}

Therefore, applying Theorem \ref{persistentlinearhyp} to the sequence $\psi_{g}$ we obtain, 

\begin{theorem} \label{persistentdomination}
Suppose that there exists an open neighborhood $V(\epsilon)$ of $(M,g)$ in $F^{2}(M,g)$. Then there exist constants $K>0$, $D \geq T_{g}$, 
$\lambda \in (0,1)$ such that:
\begin{enumerate}
\item For every periodic point $\theta$ with minimum period
$\omega \geq D$, we have
$$ \prod_{i=0}^{k-1} \parallel D\phi_{D}\vert_{E^{s}(\phi_{iD}(\theta)} \parallel \leq K\lambda^{k} $$
and
$$ \prod_{i=0}^{k-1} \parallel D\phi_{-D}\vert_{E^{u}(\phi_{-iD}(\theta)} \parallel \leq K\lambda^{k} ,$$
where $E^{s}(\tau) \oplus E^{u}(\tau) = N_{\tau}$ is the hyperbolic splitting of the geodesic flow of $(M,g)$ at a periodic
point $\tau$ and $k = [\frac{\omega}{D}]$.
\item There exists a continuous Lagrangian, invariant, dominated  splitting 
$$
T_{\theta} T_{1}M = G^{s}(\theta)\oplus G^{u}(\theta) \oplus X(\theta)
$$
 in the closure of the set of periodic orbits of $\phi_{t}$ which extends the hyperbolic splitting of periodic orbits: if $\theta$ is
periodic then $G^{s}(\theta) = E^{s}(\theta)$, $G^{u}(\theta) = E^{u}(\theta)$.
\end{enumerate}

\end{theorem}

Theorem \ref{persistentdomination} improves Theorem 2.1 in \cite{ruggiero91} where the same conclusions are claimed assuming that
the geodesic flow of $(M,g)$ is in the $C^{1}$ interior of the set of Hamiltonian flows all of whose periodic orbits are hyperbolic. 

Hence, the proof of Theorem \ref{THM1} follows from the combination of Theorems  \ref{domination} and Theorem \ref{persistentdomination}. 

\subsection{Proof of Theorem \ref{THM2}}

Let $E^{2}(M,g)$ be the set of Riemannian metrics in $M$ conformally equivalent to $(M,g)$, endowed with the $C^{2}$ topology,
whose geodesic flows are expansive. The main result of the subsection is an improved version of Proposition 1.1 in \cite{ruggiero91}.

\begin{theorem} \label{periodichyp}
The interior of $E^{2}(M,g)$ is contained in $F^{2}(M,g)$.
\end{theorem}

We just give an outline of the proof based on \cite{ruggiero91}. The argument is by contradiction. Suppose that there exists
$(M,h)$ in the interior of  $E^{2}(M,g)$ whose geodesic flow has a nonhyperbolic periodic point $\theta$. Let $\Sigma$ be a cross section of the
geodesic flow at $\theta$ tangent to $N_{\theta}$. The derivative of the Poincar\'{e} return map has some eigenvalues in the unit circle.
By the results of Rifford-Ruggiero \cite{rr12} $DP$ every generic property in the symplectic group is attained by $C^{2}$ perturbations by potentials
of $(M,h)$ preserving the orbit of $\theta$. This means that there exists $(M,\bar{h})$ $C^{2}$-close to $(M,h)$ and conformally equivalent to it
such that the orbit of $\theta$ is still a periodic orbit of the geodesic flow of $(M,\bar{h})$ and the derivative of the
Poincar\'{e} map $\bar{P}: \Sigma \longrightarrow \Sigma$ has generic unit circle eigenvalues. By the central manifold Theorem of
Hirsch-Pugh-Shub \cite{hps77} there exists a central invariant submanifold $\Sigma_{0} \subset \Sigma$
such that the return map $P_{0}$ of the geodesic flow of $(M,\bar{h})$ is tangent to the invariant subspace associated to the eigenvalues
of $D\bar{P}$ in the unit circle. Moreover, we can suppose by the $C^{k}$ Ma\~{n}\'{e}-generic version of the Klingenberg-Takens Theorem 
due to Carballo-Gon\c{c}alves \cite{kn:CG} that the Birkhoff normal form of the Poincar\'{e} map at the periodic point $\theta$ is 
generic. So we can apply the Birkhoff-Lewis fixed point Theorem due to Moser \cite{kn:Moser} to deduce that given $\delta >0$ there 
exists infinitely many closed orbits of the geodesic flow of $(M,\bar{h})$ in the $\delta$-tubular
neighborhood of the orbit of $\theta$. This clearly contradicts the expansiveness of the geodesic flow of $(M,\bar{h}) \in E^{2}(M,g)$.

In the case where $(M,g)$ is a closed surface, we know that the expansiveness of the geodesic flow implies the density of the set
of periodic orbits in the unit tangent bundle (see \cite{ruggiero91} for instance). So if $(M,g)$ is in the interior of $E^{2}(M,g)$ the closure of
the set of periodic orbits is a hyperbolic set by Theorem \ref{THM1}, and since this set is dense its closure is the unit tangent bundle and
therefore, the geodesic flow is Anosov. If the dimension of $M$ is arbitrary, then we know that if $(M,g)$ has no conjugate points, 
the expansiveness of the geodesic flow implies the density of periodic orbits as well, so we can extend the above result for surfaces.  

\subsection{Main ideas to show Theorem \ref{persistentlinearhyp}}

As mentioned before, Theorem \ref{persistentlinearhyp} is a symplectic version of Lemma II.3 in \cite{kn:Manest} that is proved for general families 
of periodic sequences of linear isomorphisms of $\mathbb{R}^{n}$.  Theorem \ref{persistentlinearhyp} has been already used in \cite{ruggiero91}, and since there is no written proof in the literature we would like to give a sketch of proof for the sake of completeness. We shall not repeat all the steps of the proof of Lemma II.3 in \cite{kn:Manest} because the arguments extend quite forwardly, we shall just point out where the symplectic assumption matters. 

The proof of Lemma II.3 in \cite{kn:Manest} has two main parts. The first part is based on the generic linear algebra of what Ma\~{n}\'{e} calls uniformly contracting 
families of periodic sequences of linear isomorphisms, namely, uniformly 
hyperbolic families of periodic sequences where the unstable part of each sequence is trivial (see \cite{kn:Manest} from pages 527 to 532). Since 
the restriction of the dynamics of a uniformly hyperbolic periodic sequence to the stable subspace gives rise to a uniformly contracting periodic 
sequence the argument consists in proving separatedly uniform contraction properties for the stable part of the dynamics and then uniform expansion properties for the 
unstable part of the dynamics. In the case of hyperbolic symplectic matrices,  the invariant subspaces of the dynamics are always 
Lagrangian, so we have the following elementary result of symplectic linear algebra: 

\begin{lemma} \label{sympdec1}
Given a symplectic matrix $S$ and a Lagrangian invariant subspace $L$ there exists an unitary matrix $U$ such that 
\begin{enumerate}
\item $S= U^{T} YU$ where $Y$ is a $2n\times 2n$ symplectic matrix formed by $n\times n$ blocks of the form 
$$ Y =\left( \begin{array}{cc}
A & B \\
0 & (A^{T})^{-1}
\end{array} \right),$$
where $A^{T}$ is the adjoint of $A$. 
\item The matrix $A$ represents the restriction of $S$ to $L$. 
\end{enumerate}
\end{lemma}

Now, symplectic matrices in $n\times n$ blocks can be characterized in terms of certain algebraic properties of their blocks. 

\begin{lemma} \label{sympdec2}
Let 
$$
S = \left( \begin{array}{cc}
A & B \\
C & D
\end{array} \right)
$$
be a $2n \times 2n$ matrix where $A, B, C, D$ are $n \times n$ blocks. The matrix $S$ is symplectic if and only  if 
\begin{enumerate}
\item $A^{T}D - C^{T}B = I$
\item The matrices $B^{T}D$ and $A^{T}C$ are symmetric. 
\end{enumerate}
So any matrix 
$$ 
M =  \left( \begin{array}{cc}
A & B \\
0 & D
\end{array} \right)
$$ 
formed by $n\times n$ blocks $A,B,0,D$ is symplectic if and only if $D= (A^{T})^{-1}$ and $B^{T}D$ is symmetric . 
\end{lemma}

Hence, to extend to the symplectic case Ma\~{n}\'{e}'s generic linear algebra arguments for uniformly contracting families of periodic isomorphisms 
one can consider the family of restrictions of hyperbolic symplectic matrices to their stable subspaces.  This family is represented by a family 
of uniformly contracting periodic $n\times n$ linear isomorphisms $\psi = \{ A^{\alpha}, \alpha \in \Lambda\}$ placed in the upper left block 
of the differentials of Poincar\'{e} maps according to Lemma \ref{sympdec1}.  
Then observe that any open neighborhood of the family $\psi$ according to the distance $d(\psi, \eta)$ can be embedded in a neighborhood 
of a family of symplectic isomorphisms just by applying Lemma \ref{sympdec2}. We can build a symplectic family of symplectic isomorphisms from a perturbation $\tilde{A}^{\alpha}$ of $A^{\alpha}$ taking $\tilde{D} = (\tilde{A}^{T})^{-1}$ and finding $\tilde{B}$ close to $B$ such that $\tilde{B}^{T}\tilde{D}$ is symmetric. Such matrix $\tilde{B}$ exists because the set of symmetric matrices is a submanifold of the set of matrices, and $B^{T}\tilde{D}$ is close to the symmetric matrix $B^{T}D$. So there exists 
$\epsilon >0$ such that the ball $V_{\epsilon}$ of radius $\epsilon$ of matrices centered at $B^{T}\tilde{D}$ meets the submanifold of symmetric matrices in an open (relative) neighborhood of $B^{T}D$ . But the multiplication of an open neighborhood $V(B^{T})$ of $B^{T}$ by $\tilde{D}$ gives an open neighborhood 
of $B^{T}\tilde{D}$ in the set of matrices. Then for a suitable choice of $V(B^{T})$ we have that $V(B^{T})D$ contains a  matrix $\tilde{B}^{T}\tilde{D}$ that is symmetric. 
\bigskip

 Therefore, Ma\~{n}\'{e}'s arguments for uniformly contracting families can be extended to the symplectic category. 
Finally, let us remark that the symplectic nature of the family implies that contraction properties 
of the norm of the restriction to the stable part under the action of the dynamics already give expansion properties for the action of the 
dynamics on the norm of the restriction to the unstable part. This yields that it is enough to consider the contracting part of the dynamics 
of a symplectic family of periodic linear isomorphisms to extend the first part of the proof of Lemma II.3 in \cite{kn:Manest} to such families. 
\bigskip

The second part of the proof deals with the angle between the invariant subspaces  of uniformly hyperbolic families (see \cite{kn:Manest} pages 532-540). 

\begin{definition}
Given two subspaces $E, S \subset \mathbb{R}^{n}$ such that $E\bigoplus S= R^{n}$, let $\measuredangle(E,S)$ be 
defined by 
$$ \measuredangle(E,S) = \parallel L \parallel^{-1}$$ 
where $L : E^{\bot} \longrightarrow E$ is such that $S= \{ v+L(v), v \in E^{\bot}\}.$ In particular, $\measuredangle(E, E^{\bot}) = \infty$. 
\end{definition}

The main goal of this part of the proof of Lemma II.3 in \cite{kn:Manest} is to show that the invariant splitting of a 
uniformly hyperbolic family is a continuous dominated splitting. The general idea of the proof of this second part is to "move" one of the invariant 
subspaces of the dynamics with perturbations of the map $L$ while keeping the other subspace unchanged. 

The proof of the continuous domination  has two steps. First of all, so show that the angle between the invariant subspaces 
must be bounded below by a positive constant (Lemma II.9 in \cite{kn:Manest} pages 532 to 534). 
This is the content of the following result whose proof we present in detail to give a sample 
of how the arguments extend to symplectic matrices. We just follow step by step Ma\~{n}\'{e}'s proof, we even respect the notations in his paper. 

\begin{lemma} \label{angle}
Let $\{ \psi^{\alpha}, \alpha \in \Lambda\}$ be a uniformly hyperbolic family of periodic sequences of 
symplectic isomorphisms of $\mathbb{R}^{2n}$. Then, there exist $\epsilon >0$, $\gamma >0$, and $n_{0} \in \mathbb{Z}^{-}$ 
such that if $\{\eta^{\alpha}, \alpha \in \Lambda\}$ is a periodically equivalent family with $d(\psi, \eta) <\epsilon$ 
then $\{ \eta^{\alpha}, \alpha \in \Lambda\}$ is hyperbolic and the angle between stable and unstable subspaces satisfies
$$\measuredangle(E^{s}_{0}(\eta^{\alpha}), E^{u}_{0}(\eta^{\alpha})) > \gamma $$
for every $\alpha \in \Lambda$ such that the minimum period of $\eta^{\alpha}$ is greater than $n_{0}$. 
\end{lemma}

\begin{proof}
 Suppose by contradiction that the statement is false. Then there would exist 
hyperbolic periodic sequences $\eta : \mathbb{Z} \longrightarrow Sp(2n,\mathbb{R})$ with arbitrarily large period $n$, such that 
\begin{enumerate}
\item $\measuredangle(E^{s}_{0}(\eta^{\alpha}), E^{u}_{0}(\eta^{\alpha}))$ is arbitrarily small,
\item For some $\alpha \in \Lambda$ the periods of $\psi^{\alpha}$ and $\eta$ coincide, 
\item $\sup_{i} \parallel \eta_{i} - \psi^{\alpha}_{i}) \parallel$ is arbitrarily small. 
\end{enumerate}
Suppose that in the coordinates of the base $E^{s}_{0}(\eta)^{\bot} \bigoplus E^{s}_{0}(\eta)$ the matrix of 
$\prod _{j=0}^{n-1}\eta_{j}$ is 
$$  M= \left( \begin{array}{cc}
A & 0 \\
P & B
\end{array} \right)$$
where $A, P, B$ are $n \times n$ matrices. By the uniform contraction property of the stable part of the dynamics 
of the family there exist $K>0$, $\lambda \in (0,1)$ such that 
$$\parallel A^{-1} \parallel \leq K\lambda^{n}$$ 
and 
$$ \parallel B \parallel \leq K\lambda^{n} .$$
Since we can choose an orthogonal change of coordinates $Q$, we have that the matrix $M= QA$ is a symplectic matrix and thus, 
by Lemma \ref{sympdec1} we get $B= (A^{T})^{-1}$. 

Let $L : E^{s}_{0}(\eta)^{\bot} \longrightarrow E^{s}_{0}(\eta)$ be such that $\{v + L(v), v \in E^{s}_{0}(\eta)^{\bot}\} = E^{u}_{0}(\eta)$. 
Since $\prod _{j=0}^{n-1}\eta_{j}(E^{u}_{0}(\eta)) = E^{u}_{0}(\eta)$ we get 
$$ LA = P + BL,$$
and therefore, 
$$ L = PA^{-1} + BLA^{-1}$$ 
and by the previous inequalities 
$$ 
\parallel L \parallel \leq \parallel PA^{-1} \parallel + K^{2} \lambda^{2n}\parallel L \parallel.
$$
For $n$ large enough, $K^{2} \lambda^{2n} \leq \frac{1}{2}$, so we have 
$$ \frac{1}{2} \parallel P A^{-1} \parallel^{-1} \leq \parallel L \parallel^{-1} = \measuredangle(E^{s}_{0}(\eta^{\alpha}), E^{u}_{0}(\eta^{\alpha})), $$
and hence the number $\parallel P A^{-1} \parallel^{-1}$ assumes arbitrarily small values by the contradiction assumption. 

Next, define a sequence $\xi : \mathbb{Z} \longrightarrow Sp(2n, \mathbb{R})$ with minimum period $n$, where 
\begin{enumerate}
\item $\xi_{i} = \eta_{i}$ for every $0 < i \leq n-1$, 
\item $\xi_{0} = \eta_{0} \left( \begin{array}{cc}
I & C \\
0 & I
\end{array} \right).$
\end{enumerate}
The matrix $\left( \begin{array}{cc}
I & C \\
0 & I
\end{array} \right)$ is symplectic for every $C$ such that $C$ is symmetric by Lemma \ref{sympdec2}. Then, 
$$ \prod_{i=0}^{n-1} \xi_{i} = 
\left( \begin{array}{cc}
A & 0 \\
P & B
\end{array} \right) \left( \begin{array}{cc}
I & C \\
0 & I
\end{array} \right) = \left( \begin{array}{cc}
A & AC \\
P & B + PC
\end{array} \right).$$
So the goal is to find a symmetric matrix $C$ with small norm such that the above matrix has an eigenvalue equal to $1$. In this way 
we get a contradiction because we are supposing that the family $\{ \psi^{\alpha}, \alpha \in \Lambda\}$ is uniformly 
hyperbolic so any sufficiently close family would have to be as well. 

To find the matrix $C$ let us consider the system 
$$ 
\left\{
\begin{array}{ccl}
Ax+ ACy & =& x \\ 
Px+(PC+B) y & = &y.
\end{array}
\right.
$$
For a solution $(x,y)$ of the system we would have 
$$ x= (I-A)^{-1}ACy= -(I-A^{-1}){-1}Cy $$ 
and 
$$ (I-B)^{-1}P(I-(I-A^{-1})^{-1})Cy=y.$$ 
Notice that $I-(I-A^{-1})^{-1}= -A^{-1}(I-A^{-1})^{-1}$, so we get 
$$ -(I-B)^{-1}PA^{-1}(I-A^{-1})^{-1}Cy = y.$$
Take a vector $v$ such that $\parallel v \parallel = \parallel PA^{-1} \parallel^{-1}$, and $\parallel PA^{-1}v \parallel =1$. 
Let 
$$y = -(I-B)^{-1}PA^{-1}v.$$ 
Since $\parallel B \parallel \leq K\lambda^{n}$ we can assume that $\parallel I - B \parallel \leq 2$. Hence $\parallel y \parallel^{-1} \leq 2$. 
Now take a vector $w$ such that 
$$ (I-A^{-1})^{-1}w = v.$$
Since the norm of $A^{-1}$ is small the matrix $(I-A^{-1}) $ is close to the identity, so we can suppose that 
$\parallel w \parallel \leq 2 \parallel v\parallel$. Next, consider a matrix $C$ such that 
$$ Cy = w, \mbox{ } \parallel C \parallel = \frac{\parallel w \parallel}{\parallel y \parallel}. $$
Observe that 
$$\parallel C \parallel \leq 4\parallel v \parallel = 4\parallel PA^{-1}\parallel^{-1} \leq 2 \measuredangle(E^{s}_{0}(\eta^{\alpha}), E^{u}_{0}(\eta^{\alpha}))$$
that can be made arbitrarily small. Thus, the matrix $C$ and the vector $y$ defined above give a fixed point $(x,y)$ for 
the matrix $ \left( \begin{array}{cc}
A & AC \\
P & B + PC
\end{array} \right)$ which shows that the sequence $\xi_{i}$ is not hyperbolic. 
\bigskip

Notice that the conditions defining $C$ are quite loose, there are many possible candidates. In particular, the matrix $C$ can be taken symmetric. Indeed, 
symmetric matrices are linear maps which send the unit sphere to ellipsoids centered at $0$. Moreover, the norm of such a map is the length of the largest axis of the 
corresponding ellipsoid. So let us consider a linear map $T$ such that $T(y)=w$ as the linear map $C$ does, and take $T$ such that 
\begin{enumerate}
\item The image of the unit vector $\frac{y}{\parallel y \parallel}$ by $T$ is $\frac{w}{\parallel y \parallel}$. 
\item The image of the unit sphere by $T$ is an ellipsoid whose largest axis is contained in the line $tw$, $t \in \mathbb{R}$, 
and whose length is $\frac{\parallel w \parallel}{\parallel y \parallel}$. 
\end{enumerate}
If we take $C= T$ we have a symmetric matrix solving the above system of equations. 
\end{proof}

The final step of the second part of the proof is to show that the uniform hyperbolicity of families combined with the 
existence of a lower bound for the angle between invariant subspaces implies the domination condition (\cite{kn:Manest} pages 534-540). The argument is 
by contradiction: if the domination condition is not satisfied then it is possible to find a small perturbation of the family such that 
the invariant subspaces of the perturbed one are very close to each other, which is impossible by Lemma \ref{angle}. 
The proof is involved but again, the tools of the proof are quite general and elementary in linear algebra, they can be 
adapted straighforwardly to symplectic matrices.

\appendix 
\section{Proof of Lemma \ref{LEMtechnical1nov}}

First of all, we observe that given $L_{\bar{i},\bar{j}}^N$, the existence of $K(N)$ follows by homogeneity and continuity of the mapping 
$$
v \in L_{\bar{i},\bar{j}}^N \, \longmapsto \, \bigl(tv_{\bar{i}}\bigr) \odot \bigl(sv_{\bar{j}}\bigr)  \in \R.
$$
Let us now demonstrate the existence of  $L_{\bar{i},\bar{j}}^N$  by induction over $N$. In fact, setting $f=v_{\bar{i}}, g=v_{\bar{j}}$, it is sufficient to show that the set $\mathcal{L}$ of $w=(f,g) \in L^2([0,1];\R^2)$ with $f$ and $g$ polynomials satisfying 
\begin{eqnarray}\label{21nov1}
\left\{
\begin{array}{rcl}
\int_0^1 f(s) ds & = & 0 \\
\int_0^1 sf(s) ds & = & 0 \\ 
\int_0^1 g(s) ds & = & 0 \\ 
\int_0^1 sg(s) ds & = & 0\\
f \odot (sg) & = & 0 \\
 g \odot (sf) & = & 0 \\
  f \odot (s^2 g) & = & 0 \\
  g \odot (s^2 f) & = & 0 \\
\end{array}
\right.
\end{eqnarray}
and
\begin{eqnarray}\label{21nov2}
(tf) \odot (sg) \ne 0
\end{eqnarray}
contains (adding the origin) vector spaces $L^N$ of any dimension. When $f,g$ are polynomials, that is of the form 
$$
f(t) = \sum_{p \in \Z} a_p t^p \quad \mbox{and} \quad g(t) = \sum_{q \in \Z} b_q t^q
$$
with $a_p=b_q=0$ for any $p,q <0$ and  $a_p=b_q=0$ for large $p,q$, we check easily that (from now on, we omit to write the set $\Z$ containing $p$ and $q$)
$$
f \odot g = \sum_{p,q } \alpha_{p,q} a_p \, b_q,
$$
with $ 1/\alpha_{p,q}=(q+1)(p+q+2)$. Then we have 
\begin{eqnarray*}
\left\{
\begin{array}{rcl}
\int_0^1 f(s) ds & =  & \sum_{p} \frac{1}{p+1} \, a_p \\
\int_0^1 sf(s) ds & = & \sum_{p} \frac{1}{p+2} \, a_p \\
\int_0^1 g(s) ds & =  & \sum_{q } \frac{1}{q+1} \, b_q \\
\int_0^1 sg(s) ds & = & \sum_{q} \frac{1}{q+2} \, b_q \\
f \odot (sg) & =  & \sum_{p,q } \alpha_{p,q} a_p \, b_{q-1} \\
f \odot (s^2 g) & = & \sum_{p,q } \alpha_{p,q} a_p \, b_{q-2} \\
g \odot (sf) & = &  \sum_{p,q } \alpha_{q,p} a_{p-1} \, b_{q} \\
g \odot (s^2 f) & = &  \sum_{p,q } \alpha_{q,p} a_{p-2} \, b_{q},
\end{array}
\right.
\end{eqnarray*}
and
$$
(tf) \odot (sg) =  \sum_{p,q } \alpha_{p,q} a_{p-1}\, b_{q-1}.
$$
We can now show that the set $\mathcal{L} \cup \{0\}$ contains a vector line. As a matter of fact, taking $f(t)=1-6t+6t^2$ and taking $g$ in the set of polynomial of degree $\leq d$, leads to the system
\begin{eqnarray}\label{13dec1}
\left\{
\begin{array}{rcl}
\sum_{q } \frac{1}{q+1} \,\, b_q & = & 0 \\
\sum_{q} \frac{1}{q+2} \,\, b_q & = & 0 \\
\sum_{q } (\alpha_{0,q}-6 \alpha_{1,q}+6 \alpha_{2,q}) \, b_{q-1} = \sum_{q } \frac{q+1}{(q+3)(q+4)(q+5)} \,\, b_{q} & = & 0 \\
\sum_{q } (\alpha_{0,q}-6 \alpha_{1,q}+6 \alpha_{2,q}) \, b_{q-2} = \sum_{q } \frac{q+2}{(q+4)(q+5)(q+6)} \, \,b_{q} & = & 0 \\
\sum_{q } (\alpha_{q,1}-6 \alpha_{q,2}+6 \alpha_{q,3})\, b_{q} = -\sum_{q} \frac{q+2}{(q+3)(q+4)(q+5)} \,\,b_{q}  & = & 0 \\
\sum_{q } (\alpha_{q,2}-6 \alpha_{q,3}+6 \alpha_{q,4}) \,b_{q} = \frac{1}{30}\sum_{q} \frac{q^2-16q-60}{(q+4)(q+5)(q+6)} \,\,b_{q}  & = & 0,
\end{array}
\right.
\end{eqnarray}
which is the system of equations of the intersection of $6$ hyperplans $H_1$, $H_2$, $H_3$, $H_4$, $H_5$ and $H_6$ respectively, that we denote by 
$$
V:=\cap_{i=1}^6 H_i \subset \R_d[X].
$$
Then $V$ has dimension at least $d-5$. We can check with Maple that there is a $\bar{d} \in \N$ sufficiently large such that $V$ is not contained in the kernel of the linear form $$
\phi: (b_q) \in \R_{\bar{d}}[X] \mapsto  \sum_{q } \frac{q+3}{(q+4)(q+5)(q+6)}  \, b_{q}. 
$$
For every $i=1,...,6$, let $\phi_i$ be the linear form correspoding to the $i$-th line in (\ref{13dec1}) and denote by $A(d)$ the $7 \times (d+1)$ matrix whose seven lines are given by the coeficients of $\phi_1$, $\phi_2$, $\phi_3$, $\phi_4$, $\phi_5$, $\phi_6$ and $\phi$ respectively, that is 
$$A(d):=
\left(
\begin{matrix}
 1 & \frac{1}{2} & \cdots &  \frac{1}{d} &  \frac{1}{d+1} \\
  \frac{1}{2} & \frac{1}{3} & \cdots & \frac{1}{d+1} &  \frac{1}{d+2} \\
  \frac{1}{60} &\frac{1}{60} & \cdots &  \frac{d}{(d+2)(d+3)(d+4)} &  \frac{d+1}{(d+3)(d+4)(d+5)} \\
    \frac{1}{60} & \frac{1}{70} & \cdots &  \frac{d+1}{(d+3)(d+4)(d+5)} &  \frac{d+2}{(d+4)(d+5)(d+6)} \\
    \frac{1}{30} & \frac{1}{40} & \cdots & \frac{d+1}{(d+2)(d+3)(d+4)} &  \frac{d+2}{(d+3)(d+4)(d+5)}  \\
       \frac{-1}{2} & \frac{-5}{14} & \cdots &  \frac{(d-1)^2-16(d-1)-60}{(d+3)(d+4)(d+5)} &  \frac{d^2-16d-60}{(d+4)(d+5)(d+6)} \\
 \frac{1}{40} & \frac{2}{105} & \cdots & \frac{d+2}{(d+3)(d+4)(d+5)} &  \frac{d+3}{(d+4)(d+5)(d+6)}         
\end{matrix}
\right).
$$
We check with Maple that $\mbox{rank} \bigl(A(50)\bigr)=7$, which shows that $\phi \notin \mbox{Span} \left( \phi_1,...,\phi_6 \right),$ and in turn that $V \not\subset \mbox{Ker} (\phi)$. Therefore there is a solution $(b_q) \in \R_{50}[X]$ of the above system which satisfies 
$$
\Bigl(f(t)=1-6t+6t^2\Bigr) \odot \left(  g(t) = \sum_{q} b_q t^q \right)  =   \sum_{q } \frac{q+3}{(q+4)(q+5)(q+6)}  \, b_{q}=1.
$$

Assume now that we proved the existence of a vector space $L^N \subset \mathcal{L} \cup \{0\}$ of dimension $N\geq 1$. Let $\{(f_1,g_1), \ldots, (f_N,g_N)\}$ be a basis of $L^N$. We need to find a pair $(f,g)$ such that for any $\alpha=(\alpha_1, \ldots, \alpha_N) \in \R^N$ and $\beta \in \R$, the pair 
$$
\left( \beta f + \sum_{l=1}^N \alpha_l f_l,  \beta g + \sum_{l=1}^N \alpha_l g_l \right)
$$
satisfies (\ref{21nov1}) and (\ref{21nov2}). By bilinearity of the $\odot$ product, this amounts to say that
\begin{eqnarray}\label{13dec2}
\left\{
\begin{array}{rcl}
\beta^2 \, f \odot (sg) + \beta \sum_{l=1}^N \alpha_l \, f \odot (sg_l) + \beta  \sum_{l=1}^N \alpha_l \, f_l \odot (sg) & =  & 0 \\
\beta^2 \, f \odot (s^2g) + \beta \sum_{l=1}^N \alpha_l \, f \odot (s^2g_l) + \beta  \sum_{l=1}^N \alpha_l \, f_l \odot (s^2g) & =  & 0 \\
\beta^2 \, g \odot (sf) + \beta \sum_{l=1}^N \alpha_l \, g \odot (sf_l) + \beta  \sum_{l=1}^N \alpha_l \, g_l \odot (sf) & =  & 0 \\
\beta^2 \, g \odot (s^2f) + \beta \sum_{l=1}^N \alpha_l \, g \odot (s^2f_l) + \beta  \sum_{l=1}^N \alpha_l \, g_l \odot (s^2f) & =  & 0,
\end{array}
\right.
\end{eqnarray}
\begin{eqnarray}\label{21nov5}
 \left\{
\begin{array}{rcl}
\int_0^1 f(s)\, ds  & = & 0 \\
\int_0^1 sf(s)\, ds   & = & 0, 
\end{array}
\right.
\qquad  
 \left\{
\begin{array}{rcl}
\int_0^1 g(s)\, ds  & = & 0 \\
\int_0^1 sg(s)\, ds  & = & 0, 
\end{array}
\right.
\end{eqnarray}
and
\begin{multline}\label{13dec3}
\beta^2 \, (tf) \odot (sg) + \beta \sum_{l=1}^N \alpha_l \, (tf) \odot (sg_l) + \beta  \sum_{l=1}^N \alpha_l \, (tf_l) \odot (sg) + \sum_{l=1}^N \alpha_l^2  (tf_l) \odot (sg_l)\\
\ne 0. 
\end{multline}
In fact, any pair $(f,g)$ satisfying (\ref{21nov1})-(\ref{21nov2}) and the systems 
\begin{eqnarray}\label{21nov3}
 \left\{
\begin{array}{rcl}
f \odot (sg_l) & = & 0 \\
f \odot (s^2 g_l)  & = & 0 \\
 g_l \odot (sf)  & = & 0 \\
g_l \odot (s^2f)   & = & 0 \\
(tf) \odot (sg_l)  & = & 0 
\end{array}
\right.
\qquad  
 \left\{
\begin{array}{rcl}
f_l \odot (sg) & = & 0 \\
 f_l \odot (s^2g) & = & 0 \\
g \odot (sf_l)  & = & 0 \\
g \odot (s^2 f_l)  & = & 0 \\
(tf_l) \odot (sg) & = & 0
\end{array}
\right.
\end{eqnarray}
provides a solution. First we claim that there is a polynomial $f_0$ satisfying the left systems in (\ref{21nov5}) and (\ref{21nov3}). As a matter of fact, $f_0$ has to belong to the intersection of $2+5N$ hyperplanes in $\R_d[X]$. Such an intersection is not trivial if $d$ is large enough. The function $f=f_0$ being fixed, we need now to find a polynomial $g$ solution to the four last equations of system (\ref{21nov1}), to (\ref{21nov2}), and to the right systems in  (\ref{21nov5}) and (\ref{21nov3}). Thus $g$ needs to belong to the intersection of $6+5N$ hyperplanes and to satisfies (\ref{21nov2}). Let 
$$
f_l(t) = \sum_{p=0}^P a_p^l t^p \quad \mbox{and} \quad g_l(t) = \sum_{q =0}^P b_q^l t^q,
$$
$$
f_0(t) = \sum_{p=0}^d a_p^0 t^p \quad \mbox{and} \quad g(t) = \sum_{q \in \Z} b_q t^q,
$$
where $P$ is the maximum of the degrees of $f_1, \ldots, f_N, g_1, \ldots, g_N$ and $d$ is the degree of $f_0$. We have 
\begin{eqnarray}\label{21nov9}
 \left\{
\begin{array}{rcl}
\int_0^1 g(s) ds & =  & \sum_{q } \frac{1}{q+1} \, b_q \\
\int_0^1 sg(s) ds & = & \sum_{q} \frac{1}{q+2} \, b_q,
\end{array}
\right.
\end{eqnarray}
\begin{eqnarray}\label{21nov6}
 \left\{
\begin{array}{rcl}
f_l \odot (sg) & = & \sum_{q}  \left( \sum_p  \alpha_{p,q+1} \, a_p^l \right) \, b_q \\
 f_l \odot (s^2g) & = & \sum_{q}  \left( \sum_p  \alpha_{p,q+2} \, a_p^l \right) \, b_q \\
g \odot (sf_l)  & = &  \sum_{q}  \left( \sum_p  \alpha_{q,p+1} \, a_p^l \right) \, b_q \\
g \odot (s^2 f_l)  & = &  \sum_{q}  \left( \sum_p  \alpha_{q,p+2} \, a_p^l \right) \, b_q,
\end{array}
\right.
\end{eqnarray}
and
\begin{equation}\label{21nov8}
(tf_l) \odot (sg)  = \sum_{q}  \left( \sum_p  \alpha_{p+1,q+1} \, a_p^l \right) \, b_q 
\end{equation}
for every $l=1,...,N,$ and moreover 
\begin{eqnarray}\label{21nov7}
 (tf_0) \odot (sg) = \sum_{q}  \left( \sum_p  \alpha_{p+1,q+1} \, a_p^0 \right) \, b_q.
\end{eqnarray}
We need to show that the kernel of the linear form (given by (\ref{21nov7}))
$$
\Phi_{f_0} \, : \, \left( b_q\right) \in \R_d[X] \, \longmapsto \,  \sum_{q}  \left( \sum_p  \alpha_{p+1,q+1} \, a_p^0 \right) \, b_q
$$
does not contain the intersection of the kernels of the $2+4(N+1)+N=5N+6$ linear forms given by (\ref{21nov9})-(\ref{21nov8}). If this is the case, for every integer $d\geq 0$, any choice of $f_0$ in $\R_d[X]$, and any integer $d'\geq 0$, there are  $C=5N+6$ real numbers (not all zero)
$$
\lambda_1^{l,d'}, \ldots, \lambda_{5}^{l,d'}, \lambda_{6}^{0,d'}, \lambda_{7}^{0,d'},  \lambda_{8}^{0,d'},  \lambda_{9}^{0,d'},  \lambda_{10}^{d'},  \lambda_{11}^{d'},  
$$
such that for every integer $q\in \{0, \ldots, d'\}$,
\begin{eqnarray*}
 \sum_{p=0}^d  \alpha_{p+1,q+1} \, a_p^0 & = & \sum_{l=1}^N \lambda_1^{l,d'}  \left( \sum_{p=0}^P  \alpha_{p,q+1} \, a_p^l \right) + \sum_{l=1}^N \lambda_2^{l,d'} \left( \sum_{p=0}^P  \alpha_{p,q+2} \, a_p^l \right) \\
 & \quad & \quad + \sum_{l=1}^N \lambda_3^{l,d'}  \left( \sum_{p=0}^P  \alpha_{q,p+1} \, a_p^l \right) + \sum_{l=1}^N \lambda_4^{l,d'}  \left( \sum_{p=0}^P  \alpha_{q,p+2} \, a_p^l \right) \\
     & \quad & \quad + \sum_{l=1}^N \lambda_5^{l,d'}  \left( \sum_{p=0}^d  \alpha_{p+1,q+1} \, a_p^l \right) +  \lambda_{6}^{0,d'}  \left( \sum_{p=0}^d  \alpha_{p,q+1} \, a_p^0 \right) \\
& \quad & \quad +  \lambda_7^{0,d'} \left( \sum_{p=0}^d  \alpha_{p,q+2} \, a_p^0 \right) +\lambda_8^{0,d'}  \left( \sum_{p=0}^d  \alpha_{q,p+1} \, a_p^0 \right)\\
 & \quad & \quad + \lambda_9^{0,d'}  \left( \sum_{p=0}^d  \alpha_{q,p+2} \, a_p^0 \right) +    \frac{\lambda_{10}^{d'}}{q+1} + \frac{ \lambda_{11}^{d'}}{q+2}.
\end{eqnarray*}
Observe that the above equality can be written as 
\begin{eqnarray*}
0 & = &   \sum_{p=0}^d  \left[  \left( \sum_{l=1}^N \lambda_5^{l,d'} \, a_p^l\right) -  a_p^0  \right] \, \alpha_{p+1,q+1}  \\
 & \quad & \quad +\sum_{p=0}^d    \left[  \left(  \sum_{l=1}^N \lambda_1^{l,d'} \, a_p^l \right)  +   \lambda_{6}^{0,d'}   \, a_p^0     \right] \, \alpha_{p,q+1} \\
  & \quad & \quad +\sum_{p=0}^d    \left[  \left(  \sum_{l=1}^N \lambda_2^{l,d'} \, a_p^l \right)  +   \lambda_{7}^{0,d'}   \, a_p^0     \right] \, \alpha_{p,q+2} \\
     & \quad & \quad +\sum_{p=0}^d    \left[  \left(  \sum_{l=1}^N \lambda_3^{l,d'} \, a_p^l \right)  +   \lambda_{8}^{0,d'}   \, a_p^0     \right] \, \alpha_{q,p+1} \\
      & \quad & \quad +\sum_{p=0}^d    \left[  \left(  \sum_{l=1}^N \lambda_4^{l,d'} \, a_p^l \right)  +   \lambda_{9}^{0,d'}   \, a_p^0     \right] \, \alpha_{q,p+2} \\
       & \quad & \quad + \frac{\lambda_{10}^{d'}}{q+1} + \frac{ \lambda_{11}^{d'}}{q+2}.
         \end{eqnarray*}
For every $q$, let 
$$
V(q) = \bigl(V^1(q), \ldots, V^{7}(q) \bigr) \in \R^{7(d+1)}
$$
with
 $$
 V^i(q) = \bigl( V_0^i(q), \ldots, V_d^i(q) \bigr) \in \R^{d+1} \qquad \forall i=1, \ldots, 7,
 $$
 defined by
$$
\left\{
\begin{array}{rcl}
V_p^1 (q) & = & \alpha_{p+1,q+1} \\
V_p^2 (q) & = & \alpha_{p,q+1} \\
V_p^3 (q) & = & \alpha_{p,q+2} \\
V_p^4 (q) & = & \alpha_{q,p+1} \\
V_p^5 (q) & = & \alpha_{q,p+2} \\
V_p^{6} (q) & = & \frac{1}{(d+1)(q+1)} \\
V_p^{7} (q) & = & \frac{1}{(d+1)(q+2)} 
\end{array}
\right.
$$
for every $p=0, \ldots, d$. The above equality means that for every $d'\geq 0$, there is a linear form $\Psi^{d'}$ on $\R^{7(d+1)}$
of the form 
\begin{eqnarray*}
\Psi^{d'} (V) & = &   \sum_{p=0}^d  \left[  \Gamma_p^{1,d'} -  a_p^0  \right] \, V_p^1  +\sum_{p=0}^d    \left[   \Gamma_p^{2,d'}   +   \lambda_{6}^{0,d'}   \, a_p^0     \right] \, V_p^2 \\
  & \quad & \quad +\sum_{p=0}^d    \left[    \Gamma_p^{3,d'}  +   \lambda_{7}^{0,d'}   \, a_p^0     \right] \, V_p^3 +\sum_{p=0}^d    \left[   \Gamma_p^{4,d'}   +   \lambda_{8}^{0,d'}   \, a_p^0     \right] \, V_p^4 \\
    & \quad & \quad +\sum_{p=0}^d    \left[   \Gamma_p^{5,d'} +   \lambda_{9}^{0,d'}   \, a_p^0     \right] \, V_p^5  +\sum_{p=0}^d   \lambda_{10}^{d'} \, V_p^6 + \sum_{p=0}^d   \lambda_{11}^{d'}  \, V_p^7
         \end{eqnarray*}
for every $V=\bigl( V^1, \ldots, V^{7}\bigr) \in (\R^{(d+1)})^{7}$ such that
$$
\Psi^{d'} \bigl(V(q)\bigr) =0 \qquad \forall q \in \{0,\ldots, d'\}.
$$
For every integer $d'\geq 0$, let $\dim(d')$ be the dimension of the vector space which is generated by $V(0), \ldots, V(d')$. The function $d' \mapsto \dim(d')$ is nondecreasing and  valued in the positive integers. Moreover it is bounded by $7(d+1)$. Thus it is stationnary and in consequence there is $\bar{d}'\geq 0$ such that for every $q >\bar{d}'$, 
$$
V(q) \in \mbox{Span} \Bigl\{ V(0), \ldots, V\bigl( \bar{d}'\bigr) \Bigr\}.
$$ 
Therefore there is a linear form $\Psi : \R^{7(d+1)} \rightarrow \R$
of the form 
\begin{eqnarray*}
\Psi (V) & = &   \sum_{p=0}^d  \left[  \Gamma_p^{1} -  a_p^0  \right] \, V_p^1  +\sum_{p=0}^d    \left[   \Gamma_p^{2}   +   \lambda_{6}^{0}   \, a_p^0     \right] \, V_p^2 \\
  & \quad & \quad +\sum_{p=0}^d    \left[    \Gamma_p^{3}  +   \lambda_{7}^{0}   \, a_p^0     \right] \, V_p^3 +\sum_{p=0}^d    \left[   \Gamma_p^{4}   +   \lambda_{8}^{0}   \, a_p^0     \right] \, V_p^4 \\
    & \quad & \quad +\sum_{p=0}^d    \left[   \Gamma_p^{5} +   \lambda_{9}^{0}   \, a_p^0     \right] \, V_p^5  +\sum_{p=0}^d   \lambda_{10} \, V_p^6 + \sum_{p=0}^d   \lambda_{11}  \, V_p^7,
         \end{eqnarray*}
for every $V=\bigl( V^1, \ldots, V^{7}\bigr) \in (\R^{(d+1)})^{7}$ such that
$$
\Psi \bigl(V(q)\bigr) =0 \qquad \forall q \in \N.
$$
We observe that for any integers $p,q \geq 0$,
$$
\alpha_{p,q}=\frac{1}{(q+1)(p+q+2)} = \frac{1}{p+1} \left[   \frac{1}{q+1} - \frac{1}{q+p+2}   \right],
$$ 
then we have for all $q\in \N$, 
\begin{eqnarray*}
0 & = & \Psi \bigl(V(q)\bigr)   \\
& = &   \sum_{p=0}^d  \left[ \frac{   \Gamma_p^{1} -  a_p^0 }{p+2}\right]  \, \left( \frac{1}{q+2}\right)  -  \sum_{p=0}^d   \left[ \frac{   \Gamma_p^{1} -  a_p^0 }{p+2}\right]  \, \left( \frac{1}{q+p+4}\right)\\
& \quad & \quad +\sum_{p=0}^d   \left[ \frac{   \Gamma_p^{2}   +   \lambda_{6}^{0}   \, a_p^0 }{p+1}\right]  \, \left( \frac{1}{q+2}\right)    -  \sum_{p=0}^d   \left[ \frac{   \Gamma_p^{2}   +   \lambda_{6}^{0}   \, a_p^0   }{p+1}\right]  \, \left( \frac{1}{q+p+3}\right)   \\
& \quad & \quad +\sum_{p=0}^d   \left[ \frac{   \Gamma_p^{3}  +   \lambda_{7}^{0}   \, a_p^0 }{p+1}\right]  \, \left( \frac{1}{q+3}\right)    -  \sum_{p=0}^d   \left[ \frac{   \Gamma_p^{3}  +   \lambda_{7}^{0}   \, a_p^0   }{p+1}\right]  \, \left( \frac{1}{q+p+4}\right)   \\
    & \quad & \quad  +\sum_{p=0}^d    \left[  \frac{ \Gamma_p^{7} +   \lambda_{8}^{0}   \, a_p^0 }{ p+2}    \right] \, \left( \frac{1}{q+p+3}\right)+\sum_{p=0}^d    \left[ \frac{  \Gamma_p^{5} +   \lambda_{9}^{0}   \, a_p^0 }{p+3 }    \right] \,  \left( \frac{1}{q+p+4}\right)\\
 & \quad & \quad    +\sum_{p=0}^d   \frac{\lambda_{10}}{d+1} \, \left( \frac{1}{q+1} \right)  + \sum_{p=0}^d   \frac{\lambda_{11}}{d+1} \,\left( \frac{1}{q+2} \right).
\end{eqnarray*}
This can be written as 
\begin{eqnarray*}
0 & = & \Psi \bigl(V(q)\bigr)   \\
& = &  \sum_{p=0}^d   \frac{\lambda_{10}}{d+1} \, \left( \frac{1}{q+1}\right) + \sum_{p=0}^d  \left[ \frac{   \Gamma_p^{1} -  a_p^0 }{p+2} + \frac{   \Gamma_p^{2}   +   \lambda_{6}^{0}   \, a_p^0 }{p+1} + \frac{\lambda_{11}}{d+1}  \right]  \, \left( \frac{1}{q+2}\right) \\
& \quad & \quad +  \sum_{p=0}^d  \left( \left[ \frac{   \Gamma_p^{3}  +   \lambda_{7}^{0}   \, a_p^0 }{p+1}\right]  +  \left[ \frac{  \Gamma_0^{7}  +   \lambda_{8}^{0}   \, a_0^0 }{2}\right]-\left[ \Gamma_0^{2}  +   \lambda_{6}^{0}   \, a_0^0 \right] \right)\, \left( \frac{1}{q+3}\right) \\
& \quad & -  \sum_{r=4}^{d+3}  \Delta_r \cdot \left( \frac{1}{q+r}\right)\\
    & \quad & \quad  -  \left(   \left[ \frac{   \Gamma_d^{3} + \lambda_{7}^{0} \, a_d^0  }{d+1}\right]   -  \left[ \frac{   \Gamma_d^{5}  +   \lambda_{9}^{0}   \, a_d^0   }{d+3}\right] + \left[ \frac{   \Gamma_d^{1} -   \, a_d^0   }{d+2}\right] \right) \left( \frac{1}{q+d+4}\right),
\end{eqnarray*}
where for any $r\in \{4, \ldots, d+3\}$,
\begin{multline*}
\Delta_r :=  \frac{\Gamma_{r-3}^{7} + \lambda_{8}^{0} \, a_{r-3}^0 + \Gamma_{r-4}^{5} + \lambda_{9}^{0} a_{r-4}^0}{r-1}\\
 -\frac{\Gamma_{r-3}^{2} + \Gamma_{r-4}^{1}+ \lambda_{6}^{0} \, a_{r-3}^0 - a_{r-4}^0}{r-2} -\frac{\Gamma_{r-4}^{3} + \lambda_{7}^{0} \, a_{r-4}^0}{r-3} \\
=  \frac{\Gamma_{r-3}^{7} + \Gamma_{r-4}^{5}}{r-1}  - \frac{ \Gamma_{r-4}^{1}+ \Gamma_{r-3}^{2} }{r-2}  - \frac{   \Gamma_{r-4}^{3} }{r-3} + \left( \frac{\lambda_{8}^{0}}{r-1}-\frac{\lambda_{6}^{0}}{r-2} \right) \,  a_{r-3}^0 \\
+ \left( \frac{\lambda_{9}^{0}}{r-1}-\frac{\lambda_{7}^{0}}{r-3} - \frac{1}{r-2} \right) \,  a_{r-4}^0.
\end{multline*}
The function $\Psi$ is a rational function with infinitely many zeros, so it vanishes everywhere and in consequence all its coefficients vanish. Remember in addition that by construction,
$$
\Gamma_p^l = 0 \qquad \forall p\in \{P+1,\ldots, d\}, \, \forall l \in \{1, \ldots, N\}.
$$
Then we have $\Delta_r =0$ for any $r\in \{P+5, \ldots, d+3\}$,
 that is

$$
 \left( \frac{\lambda_{8}^{0}}{r-1}-\frac{\lambda_{6}^{0}}{r-2} \right) \,  a_{r-3}^0 + \left( \frac{\lambda_{9}^{0}}{r-1}-\frac{\lambda_{7}^{0}}{r-3}- \frac{1}{r-2} \right) \,  a_{r-4}^0 =0,
$$
and in addition the coefficient in front of $ \frac{1}{q+d+4}$ vanishes, that is
$$
    \left( \frac{    \lambda_{7}^{0}   }{d+1} - \frac{ \lambda_{9}^{0} }{d+3}\right)      \, a_d^0       =  \frac{a_d^0}{d+2}.
$$
In conclusion, if there is no vector space of dimension $N+1$ in $\mathcal{L}\cup \{0\}$, then for every polynomial $f_0\in \R_d[X]$ of degree $d$ (that is $a_d^0\neq 0$) the linear form $\Phi_{f_0}$ contains the intersection of the kernels of the $5N+6$ linear forms given by (\ref{21nov9})-(\ref{21nov8}). By the above discussion, this implies that there are four reals numbers $A, B, C, D$ not all zero (because $a_d^0\neq 0$) such that 
$$
 \left(     \frac{   A   }{r-1 } + \frac{   B }{ r-2}  \right) a_{r-3}^0 + \left(   \frac{C }{ r-1} -  \frac{ 1  }{r-2} +    \frac{  D }{ r-3}    \right) a_{r-4}^0 =0,
 $$
 for any  $r\in \{P+5, \ldots, d+3\}$ and in addition
$$
    \left( \frac{   D }{d+1} + \frac{    C     }{d+3}\right)      \, a_d^0       = - \frac{  a_d^0 }{d+2} \, \Longrightarrow \, D= - (d+1) \left( \frac{1}{d+2} + \frac{C}{d+3}  \right).
$$
Note that for every $p \in \{P+2, \ldots, d\}$,
\begin{eqnarray*}
& \quad &   \frac{C }{ p+2} -  \frac{ 1  }{p+1} +    \frac{  D }{ p}  \\
 & = &  \frac{C }{ p+2} -  \frac{ 1  }{p+1} - \frac{d+1}{p} \left( \frac{1}{d+2} + \frac{C}{d+3}  \right) \\
  & = & \frac{2C(d+2)(p-d-1)(p+1)-(d+3)((2d+3)p+d+1)(p+2)}{p(p+1)(p+2)(d+2)(d+3)} 
\end{eqnarray*}
This means that the set of coefficients $(a_p^0)_{p\in \{P+1,d\}}$ belongs to the algebraic set $\mathcal{S}$ of $(d-P)$-tuples $(a_p)_{p\in \{P+1,d\}} \in  \R^{d-N} $ for which there is $(A,B,C)\in \R^3$ such that
\begin{multline}\label{constraint}
\left( \frac{2C(d+2)(p-d-1)(p+1)-(d+3)((2d+3)p+d+1)(p+2)}{p(p+1)(p+2)(d+2)(d+3)}\right) \, a_{p-1} \\
 +  \left(     \frac{   A   }{p+2 } + \frac{   B }{ p+1}  \right) \, a_p =0 \qquad \forall p \in \{P+2, \ldots, d\}.
\end{multline}
For every triple $(A,B,C) \in \R^3$, denote by $\mathcal{S}(A,B,C)$ the algebraic set of  $(d-P)$-tuples $(a_p)_{p\in \{P+1,d\}} \in  \R^{d-N} $ satisfying (\ref{constraint}). Notice that for every $(A,B,C) \in \R^3$, the function
$$
p \in  \{P+2, \ldots, d\} \, \longmapsto \,  \frac{   A   }{p+2 } + \frac{   B }{ p+1} = \frac{ (A+B)p + (A+2B)}{(p+1)(p+2)}
$$
vanishes for at most one $p$ in  $\{P+2, \ldots, d\}$. This means that given $(A,B,C) \in \R^3$ either we have
$$
a_p = C_{p}^{d} \, a_{p-1} \qquad \forall p \in \{P+2, \ldots, d\},
$$
with
\begin{multline*}
 C_{p}^{d} :=    \left( \frac{2C(d+2)(p-d-1)(p+1)-(d+3)((2d+3)p+d+1)(p+2)}{p(p+1)(p+2)(d+2)(d+3)}\right) \\
 \Big/  \left(\frac{ (A+B)p + (A+2B)}{(p+1)(p+2)}  \right)       \qquad \forall p \in \{P+2, \ldots, d\},
\end{multline*}
 or there is $\bar{p}=\bar{p}(A,B,C) \in  \{P+2, \ldots, d\}$ such that 
$$
a_p = C_{p}^{d} \, a_{p-1} \qquad \forall p \in \{P+2, \ldots, d\} \setminus \bigl\{\bar{p}\bigr\}
$$
and
$$
\left( \frac{2C(d+2)(\bar{p}-d-1)(\bar{p}+1)-(d+3)((2d+3)\bar{p}+d+1)(\bar{p}+2)}{\bar{p}(\bar{p}+1)(\bar{p}+2)(d+2)(d+3)}\right) \, a_{\bar{p}-1} =0.
$$
Since the sets we are dealing with are algebraic (see \cite{bcr98,coste82}), we infer that given $(A,B,C) \in \R^3$, the algebraic set   $\mathcal{S}(A,B,C)\subset  \R^{d-N} $ has at most dimension three, which means that  $\mathcal{S}\subset  \R^{d-N} $ has at most dimension six.

In conclusion, the coefficients  $(a_p^0)_{p\in \{0,d\}}$ of $f_0$ have to belong to the intersection of $2+5N$ hyperplanes in $\R_d[X]$, and if in addition if there is no vector space of dimension $N+1$ in $\mathcal{L}\cup \{0\}$, then the 
 $(d-P)$-tuples $(a_p^0)_{p\in \{P+1,d\}}$ must belong to $\mathcal{S}\subset  \R^{d-N} $ of dimension $\leq 6$. But, for $d$ large enough,  the intersection of $2+5N$ hyperplanes in $\R_d[X]$ with the complement of an algebraic set of dimension at most $6+P+1$ is non empty. This concludes the proof of Lemma \ref{LEMtechnical1nov}.

\signal
\signlr
\signrr

\end{document}